\documentclass{article} %
\usepackage{iclr2024_conference,times}

\usepackage{amsmath,amsfonts,bm}

\def\eqref#1{equation~\ref{#1}}

\def\1{\bm{1}}

\DeclareMathAlphabet{\mathsfit}{\encodingdefault}{\sfdefault}{m}{sl}
\SetMathAlphabet{\mathsfit}{bold}{\encodingdefault}{\sfdefault}{bx}{n}

\newcommand{\E}{\mathbb{E}}

\DeclareMathOperator*{\argmin}{arg\,min}

\usepackage{hyperref}
\usepackage{url}
\PassOptionsToPackage{numbers, compress}{natbib}

\usepackage[utf8]{inputenc} %
\usepackage[T1]{fontenc}    %
\usepackage{xr-hyper}       %
\usepackage{hyperref}       %
\usepackage{url}            %
\usepackage{booktabs}       %
\usepackage{amsfonts}       %
\usepackage{nicefrac}       %
\usepackage{microtype}      %
\usepackage{color}
\usepackage{mathtools,amssymb,amsthm}
\usepackage{amsmath}

\usepackage{thmtools}
\usepackage{thm-restate}

\usepackage{natbib}

\usepackage{cleveref}

\crefformat{equation}{(#2#1#3)}

\usepackage{color}
\usepackage{comment}
\usepackage{algorithm,algpseudocode}
\usepackage{microtype}

\newcommand{\bE}{{\mathbb{E}}}

\newcommand{\bR}{{\mathbb{R}}}

\newcommand{\cD}{{\mathcal{D}}}

\newcommand{\cZ}{{\mathcal{Z}}}

\newcommand{\cO}{{\mathcal{O}}}
\newcommand{\cA}{{\mathcal{A}}}

\newcommand{\norm}[1]{\left\lVert#1\right\rVert}

\newtheorem{example}{Example}		%
\newtheorem{theorem}{Theorem}		%
\newtheorem{corollary}{Corollary}		%
\newtheorem{definition}{Definition}

\newtheorem{question}{Question}
\newtheorem{remark}{Remark}

\newcommand{\redvargrad}{\Tilde{\nabla}}

\newcommand{\prev}{\text{prev}}

\makeatother

\iclrfinalcopy

\title{Efficient Continual Finite-Sum Minimization}

\author{Ioannis Mavrothalassitis\thanks{Equal contribution}, Stratis Skoulakis\footnotemark[1], Leello Tadesse Dadi, Volkan Cevher \\
LIONS, École Polytechnique Fédérale de Lausanne\\
\{ioannis.mavrothalassitis, efstratios.skoulakis, leello.dadi, volkan.cevher\}@epfl.ch\\
}

\begin{document}

\maketitle

\begin{abstract}
Given a sequence of functions $f_1,\ldots,f_n$ with $f_i:\mathcal{D}\mapsto \mathbb{R}$, finite-sum minimization seeks a point ${x}^\star \in \mathcal{D}$ minimizing $\sum_{j=1}^nf_j(x)/n$. In this work, we propose a key twist into the finite-sum minimization, dubbed as \textit{continual finite-sum minimization}, that asks for a sequence of points ${x}_1^\star,\ldots,{x}_n^\star \in \mathcal{D}$ such that each ${x}^\star_i \in \mathcal{D}$ minimizes the prefix-sum $\sum_{j=1}^if_j(x)/i$. Assuming that each prefix-sum is strongly convex, we develop a first-order continual stochastic variance reduction gradient method ({\small \sc{CSVRG}}) producing an $\epsilon$-optimal sequence with $\Tilde{\mathcal{O}}(n/\epsilon^{1/3} + 1/\sqrt{\epsilon})$ overall \textit{first-order oracles} (FO). An FO corresponds to the computation of a single gradient $\nabla f_j(x)$ at a given $x \in \mathcal{D}$ for some $j \in [n]$. Our approach significantly improves upon the $\mathcal{O}(n/\epsilon)$ FOs that $\mathrm{StochasticGradientDescent}$ requires and the $\mathcal{O}(n^2 \log (1/\epsilon))$ FOs that state-of-the-art variance reduction methods such as $\mathrm{Katyusha}$ require. We also prove that there is no natural first-order method with $\mathcal{O}\left(n/\epsilon^\alpha\right)$ gradient complexity for $\alpha < 1/4$, establishing that the first-order complexity of our method is nearly tight.
\end{abstract}

\section{Introduction}
Given $n$ data points describing a desired input and output relationship of a model, a cornerstone task in supervised machine learning (ML) is selecting the model's parameters enforcing data fidelity. In optimization terms, this task corresponds to minimizing an objective with the \textit{finite-sum structure}.

\textbf{Finite-sum minimization:}
Given a sequence of functions $f_1,\ldots,f_n$ with $f_i:\mathcal{D}\mapsto \mathbb{R}$, let $x^\star \in \arg \min_x g(x) := \sum_{i=1}^n f_i(x)/n$. Given an accuracy target $\epsilon>0$, \textit{finite-sum minimization} seeks an approximate solution $\hat{x}$, such that: 
\begin{equation} \label{eq:finite_sum} 
    g(\hat{x}) - g(x^\star) \le \epsilon
\end{equation}
we call such an $\hat{x}$ an $\epsilon$-accurate solution. If $\hat{x}$ is random then \cref{eq:finite_sum} takes the form $\E\left[g(\hat{x})\right] - g(x^\star) \le \epsilon$.

In contemporary machine learning applications, $n$ can be in the order of billions, which makes it clear that methods tackling finite-sum minimization must efficiently scale with $n$ and the accuracy $\epsilon >0$.

First-order methods have been the standard choice for tackling
\cref{eq:finite_sum} due to their efficiency and practical behavior. The complexity of a first-order method is captured through the notion of overall \textit{first-order oracles} (FOs) where a first-order oracle corresponds to the computation of a single gradient $\nabla f_i(x)$ for some $i \in [n]$ and a point $x \in \mathcal{D}$. In case each function $f_i$ is strongly convex, $\mathrm{Stochastic}$ $\mathrm{Gradient}$ $\mathrm{Descent}$ requires $\mathcal{O}(1/\epsilon)$ FOs \cite{robbins1951stochastic,kiefer1952stochastic} so as to solve~\cref{eq:finite_sum}. Over the last years, the so-called variance reduction methods are able to solve strongly convex finite-sum problems with $\mathcal{O}(n\log(1/\epsilon))$ FOs \cite{NLST17,johnson2013accelerating,X14,DBL14,RSB12,allen2017katyusha}.

\underline{Continual Finite-Sum Minimization:}
In many modern applications new data constantly arrives over time. Hence it is important that a model is constantly updated so as to perform equally well both on the past and the new data~\cite{castro2018end,RT18incremental,H22}. Unfortunately updating the model only with respect to the new data can lead to a vast deterioration of its performance over the past data. The latter phenomenon is known in the context of \textit{continual learning} as \textit{catastrophic forgetting}~\cite{castro2018end,GMDCB13,kirkpatrick2017overcoming,mccloskey1989catastrophic}. At the same time, retraining the model from scratch using both past and new data comes with a huge computational burden.
Motivated by the above, we study a twist of~\cref{eq:finite_sum}, called \textit{continual finite-sum minimization}.

\textbf{Continual Finite-Sum Minimization:} Given a sequence of functions $f_1,\ldots,f_n$ with $f_i:\mathcal{D} \mapsto \mathbb{R}$, let $x_i^\star \in \arg\min_{x\in \cD} g_i(x) := \sum_{j=1}^i f_j(x)/i$ for all $i \in [n]$. Given an accuracy target $\epsilon>0$, \textit{continual finite-sum minimization} seeks a sequence of approximate solutions $\hat{x}_1,\ldots,\hat{x}_n \in \mathcal{D}$, such that:
\begin{equation}\label{d:instance-optimal}
g_i(\hat{x}_i) \leq g_i(x_i^\star) + \epsilon~~~\text{for each }i\in [n]
\end{equation}
We call such a sequence an $\epsilon$-optimal sequence. If $x_i^\star$ is random then \cref{d:instance-optimal} takes the form $\bE[g_i(\hat{x}_i)]-g_i(x_i^\star)\leq \epsilon$
\begin{remark}
Throughout the paper we assume that each function $f_i(\cdot)$ is strongly convex, smooth and that $\mathcal{D}$ is convex and bounded. See Section~\ref{s:prelims} for the exact definitions.
\end{remark}
Notice that $f_i(\cdot)$ can capture a new data point meaning that~\cref{d:instance-optimal} guarantees that, the model (parameterized by $\hat{x}_i$) is well-fitted in all of the data seen so far for each stage $i \in [n]$. As already discussed, since $n$ can be in the order of millions, it is important to design first-order methods
for continual finite-sum minimization that efficiently scale with $n$ and $\epsilon > 0$ with respect to overall FOs.

\vspace{-2mm}

A first attempt for tackling~\cref{d:instance-optimal} is via existing first-order methods for finite-sum minimization. As discussed above $\mathrm{StochasticGradientDescent}$ can guarantee accuracy $\epsilon > 0$ for~\cref{eq:finite_sum} using $\mathcal{O}(1 / \epsilon)$. As a result, at each stage $i \in [n]$ one could perform $\mathcal{O}(1 / \epsilon)$ stochastic gradient decent steps to guarantee accuracy $\epsilon >0$. However the latter approach would require overall $\mathcal{O}(n/\epsilon)$ FOs. On the other hand one could use a VR methods to select $\hat{x}_i \in \mathcal{D}$ at each stage $i \in [n]$. As already mentioned such methods require $\mathcal{O}(i\log(1/\epsilon))$ FOs to guarantee accuracy $\epsilon>0$ at stage $i \in [n]$. Thus the naive use of a VR method such as 
$\mathrm{SVRG}$~\cite{johnson2013accelerating,X14}, $\mathrm{SAGA}$~\cite{DBL14}, $\mathrm{SARAH}$~\cite{NLST17} and $\mathrm{Katyusha}$~\cite{allen2017katyusha} would require overall $\mathcal{O}(n^2 \log(1/\epsilon))$ FOs.

We remark that for large values of $n$ and even mediocre accuracy $\epsilon>0$ both $\mathcal{O}(n/\epsilon)$ and especially $\mathcal{O}(n^2 \log (1/\epsilon))$ are prohibitely large. As a result, the following question arises:
\begin{question}\label{q:1}
Are there first-order methods for continual finite-sum minimization whose required number of FOs scales more efficiently with respect to $n$ and $\epsilon>0$?
\end{question}

\textbf{Our Contribution} The main contribution of this work is the design of a first-order method for Continual finite-sum minimization, called $\mathrm{CSVRG}$, with overall $\tilde{\mathcal{O}}\left(n/\epsilon^{1/3} + \log n /\sqrt{\epsilon}\right)$ FO complexity. The latter significantly improves upon the $\mathcal{O}(n/\epsilon)$ FO complexity of $\mathrm{SGD}$ and the $\mathcal{O}(n^2 \log(1/\epsilon))$ complexity of variance reduction methods. As a naive baseline, we also present a \textit{sparse}  version of $\mathrm{SGD}$ tailored~\cref{d:instance-optimal} with $\mathcal{O}(\log n /\epsilon^2)$ FO complexity. The latter improves on the $n$ dependence of $\mathrm{SGD}$ with the cost of a worse dependence on $1/\epsilon$. 
 
Since variance-reduction methods such as $\mathrm{Katyusha}$ are able to achieve $\mathcal{O}(n^2 \log (1/\epsilon))$ FO complexity. A natural question is whether there exists a first-order method maintaining the $\log (1/\epsilon)$ dependence with a \textit{subquadratic dependence} on $n$, $\mathcal{O}\left(n^{2 -\alpha} \log (1/\epsilon)\right)$. We prove that the latter goal cannot be achieved for a wide class of first-order methods that we call \textit{natural first-order methods}. More precisely, we establish that there is no natural first-order method for~\cref{d:instance-optimal} requiring $\mathcal{O}(n^{2-\alpha} \log(1/\epsilon))$ FOs for any $\alpha >0$. We also establish that there is no natural first-order method for~\cref{d:instance-optimal} requiring $\tilde{\mathcal{O}}\left(n/ \epsilon^{\alpha}\right)$  FOs for any $\alpha < 1/4$. The latter lower bound implies that the $\tilde{\mathcal{O}}\left(n/ \epsilon^{1/3}\right)$ FO complexity of $\mathrm{CSVRG}$ is close to being tight. Table~\ref{t:2} summarizes our results.  
\begin{table}[!ht]
    \begin{center}
    \begin{tabular}{|c||c|}
        \hline
        \multicolumn{2}{|c|}{$\epsilon$-accuracy at each stage $i \in [n]$}\\
        \hline
         Method & Number of FOs \\
         \hline        
        $\mathrm{StochasticGradientDescent}$ & $\mathcal{O}(\frac{1}{\mu}\textcolor{blue}{n/\epsilon})$ \\
        $\mathrm{SVRG}/\mathrm{SAGA}/\mathrm{SARAH}$ & $\mathcal{O}\left(\textcolor{blue}{n^2 \log(1/\epsilon)}+ \frac{L}{\mu} \cdot \textcolor{blue}{n \log(1/\epsilon)} \right) $  \\
        \smallskip
         $\mathrm{Katyusha}$ & $\mathcal{O}\left(\textcolor{blue}{n^2 \log(1/\epsilon)} +   \sqrt{\frac{L}{\mu}} \cdot \textcolor{blue}{n^{3/2} \log(1/\epsilon)} \right)$  \\

        $\mathrm{SGD}$-$\mathrm{sparse}$ (\textcolor{red}{naive baseline}) & $\mathcal{O}\left(\frac{|\mathcal{D}|G^3}{\mu}\textcolor{blue}{1/\epsilon^2}\right)$ \\
         
         $\mathrm{CSVRG}$ (\textcolor{red}{this work}) & $\cO\left(\frac{L^{2/3}G^{2/3}}{\mu} \cdot \textcolor{blue}{(n \log n) / \epsilon^{1/3}} +\frac{L^2G}{\mu^{5/2}} \cdot \textcolor{blue}{\log n/\sqrt{\epsilon}} \right)$ \\

        Lower Bound (\textcolor{red}{this work}) & \textcolor{red}{$\Omega\left(n^{2} \log (1/\epsilon)\right)$} \\
        Lower Bound (\textcolor{red}{this work}) & \textcolor{red}{$\Omega\left(n /\epsilon^{1/4}\right)$} \\
         \hline
    \end{tabular}
    \end{center}
    \caption{Number of FOs for the strongly convex case for an $\epsilon$-accurate solution for each stage $i \in [n]$. }
    \label{t:2}
\end{table}

\textbf{Comparing the methods for $\bm{\Theta(1/n)}$ accuracy} 
According to the level of accuracy $\epsilon>0$ that~\cref{d:instance-optimal} requires, different methods might be more efficient than others. We remark that for the accuracy regime of $\epsilon = \mathcal{O}(1/n)$ all previous first-order methods require $\mathcal{O}(n^2)$ FOs while $\mathrm{CSVRG}$ requires only $\tilde{\mathcal{O}}(n^{4/3})$ FOs! 

We remark that the accuracy regime $\epsilon = \mathcal{O}(1/n)$ for~\cref{d:instance-optimal} is of particular interest and has been the initial motivation of this work. The reason for the latter lies on the generalization bounds in the strongly convex case. More precisely, the statistical error of empirical risk minimization in the strongly convex case is $\Theta(1/n)$ \cite{SS10}. Thus the accuracy 
 of the respective finite-sum problem should be selected as $\epsilon = \mathcal{O}(1/n)$ so as to match the unavoidable statistical error (see \cite{BB07} for the respective discussion). For accuracy $\epsilon = \Theta(1/n)$, a very interesting open question is bridging the gap between the $\mathcal{O}(n^{4/3})$ FO complexity of $\mathrm{CSVRG}$ and the $\Omega(n^{5/4})$ FO complexity that our lower bound suggests (setting $\epsilon =1/n$ and $\alpha = 1/4$).

\textbf{Our Techniques} 
As already discussed the naive use of a VR method at each stage $i\in [n]$ results in overall $\mathcal{O}(n^2 \log(1/\epsilon))$ FOs. A first approach in order to alleviate the latter phenomenon is to sparsely use such a method across the $n$ stages and in most intermediate stages compute $\hat{x}_i \in \mathcal{D}$ via $\mathcal{O}(1)$ FOs using a stochastic gradient estimator. However both the \textit{sparsity level} for using a VR method as well as the stochastic gradient estimator are crucial design choices for establishing accuracy $\epsilon >0$ at each stage $i \in [n]$.

Our first-order method ($\mathrm{CSVRG}$) adopts a variance-reduction pipeline along the above lines. $\mathrm{CSVRG}$ (Algorithm~\ref{alg:1}) maintains a direction $\redvargrad_{i}$ that at stage $i\in [n]$ tries to approximate $\sum_{k=1}^i \nabla f_k(\hat{x}_{i-1})/i$. Since computing the latter direction at each stage $i \in [n]$ would lead to $\Omega(n^2)$ FOs, $\mathrm{CSVRG}$

\begin{itemize}
\item Sets $\redvargrad_{i} \leftarrow \sum_{k=1}^i \nabla f_k(\hat{x}_{i})/i$ only in a very sparse sub-sequence of stages ($i$ FOs)
\item Sets $\redvargrad_{i} \leftarrow (1-\frac{1}{i})\redvargrad_{i} + \frac{1}{i}\nabla f_i(\hat{x}_{\prev})$ in most stages ($1$ FO)
\end{itemize}
In order to output an $\epsilon$-accurate point $\hat{x}_i \in \mathcal{D}$ at stage $i \in [n]$, we initialize an internal subroutine (Algorithm~\ref{alg:frequent}) to the previous output $\hat{x}_{i-1} \in \mathcal{D}$ and perform $T_i$ stochastic gradient descent steps using the following novel gradient estimator
\[\nabla_i^t \leftarrow \left(1 - \frac{1}{i}\right)\left(\nabla f_{u_t}(x_i^t)-\nabla f_{\textcolor{blue}{u_t}}(\textcolor{red}{\hat{x}_{\prev}})+\redvargrad_{i-1}\right) + \frac{1}{i}\nabla f_i(x_i^t)\]
where index $u_t$ is selected uniformly at random in $[i-1]$ and $\prev$ is the latest stage $\ell \leq i-1$ at which $\redvargrad_\ell =\sum_{k=1}^\ell \nabla f_k(\hat{x}_{\ell})/\ell$. Notice that the estimator $\nabla_i^t$ requires only $3$ FOs. Combining all the latter, we establish that by an appropriate parametrization on the sparsity of the sub-sequence at which $\redvargrad_\ell =\sum_{k=1}^\ell \nabla f_k(\hat{x}_{\ell})/\ell$ and an appropriate selection of iteration $T_i$, our first-order method $\mathrm{CSVRG}$ requires overall $\mathcal{O}(n\log n/\epsilon^{1/3} )$ FOs.

\subsection{Related Work}

Apart from the close relation of our work with the long line of research on variance reduction methods (see also Appendix~\ref{app:relatedwork} for a detailed discussion), the motivation of our work also relates with the line of research in incremental learning, the goal of which is adapting a model to new information, data or tasks without forgetting the ones that it was trained on earlier. The phenomenon of losing sight of old information is called \textit{catastrophic forgetting}~\cite{castro2018end,GMDCB13,kirkpatrick2017overcoming,mccloskey1989catastrophic,mermillod2013stability} and it is one of the main challenges of incremental learning. There have been three main empirical approaches to tackle catastrophic forgetting, regularization based \cite{nguyen2017variational}, memory based \cite{tulving1985many}, architecture based \cite{yoon2017lifelong}, as well as combination of the above \cite{sodhani2020toward}.%

\section{Preliminaries}\label{s:prelims}
In this section we introduce some basic definitions and notation.  We denote with $\mathrm{Unif}(1,\ldots,n)$ the uniform distribution over $\{1,\ldots,n\}$ and $[n]:=\{1,\ldots,n\}$.
\begin{definition}[Strong Convexity]
    A differentiable function $f:\mathcal{D} \mapsto \bR$ is $\mu$-strongly convex in $\mathcal{D}$ if and only if for all $x,y \in \mathcal{D}$
    \begin{equation}\label{eq:StrongConvex}
    f(x) \geq f(y) + \nabla f(y)^\top(x-y) + \frac{\mu}{2}\norm{x-y}^2
\end{equation}
\end{definition}
In Problem~\cref{d:instance-optimal} we make the assumption that $\mathcal{D}$ is convex and compact. We denote with $|\mathcal{D}|$ the diameter of $\mathcal{D}$, $|\mathcal{D}| = \max_{x,y \in \mathcal{D}}\norm{x-y}_2$. The compactness of $\mathcal{D}$ also provides us with the property that each $f_i:\mathcal{D} \mapsto \mathbb{R}$ is also $G$-Lipschitz and $L$-smooth. More precisely,
\begin{itemize}
\item $|f(x)-f(y)|\leq G \cdot \norm{x-y}$~~ ($G$-Lipschtiz)
\item $\norm{\nabla f(x)-\nabla f(y)}\leq L \cdot \norm{x-y}$~~~ ($L$-smooth)
\end{itemize}
where $G = \max_{x \in \mathcal{D}}\norm{\nabla f(x)}_2$ and 
$L = \max_{x \in \mathcal{D}}\norm{\nabla^2 f(x)}_2$. To simplify notation we denote as $g_i(x)$ the \textit{prefix-function} of stage $i\in [n]$,
\begin{equation}\label{eq:prefix}
g_i(x):= \sum_{j=1}^i f_j(x)/i
\end{equation}
We denote with $x_i^\star \in \mathcal{D}$ the minimizer of the function $g_i(x)$, $x_i^\star := \argmin_{x \in \mathcal{D}} g_i(x)$.
We denote the projection of $x\in \mathbb{R}^d$ to the set $\mathcal{D}$ as $\Pi_{\mathcal{D}}\left(x\right):= \argmin_{z\in \mathcal{D} }\norm{x -z}_2$. For a set $S \subseteq \mathbb{R}^d$, we denote $\Pi_{\mathcal{D}}(S) := \{ \Pi_{\mathcal{D}}(x)~:~\text{for all } x \in S\}$. Throughout the paper we assume that each function $f_i(\cdot)$ in continual finite-sum minimization~\cref{d:instance-optimal} is $\mu$-strongly convex, $L$-smooth and $G$-Lipschitz. 

\section{Our Results}
We first state our main result establishing the existence of a first-order method for continual finite-sum minimization~\cref{d:instance-optimal} with $\tilde{\mathcal{O}}\left(n/\epsilon^{1/3} + \log n/\sqrt{\epsilon}\right)$ FO complexity.
\begin{restatable}{theorem}{static}\label{thm:1}
There exists a first-order method, $\mathrm{CSVRG}$ (Algorithm~\ref{alg:1}), for continual finite-sum minimization~\cref{d:instance-optimal} with
$\cO\left(\frac{L^{2/3}G^{2/3}}{\mu} \cdot \frac{n \log n}{\epsilon^{1/3}} +\frac{L^2G}{\mu^{5/2}} \cdot \frac{\log n}{\sqrt{\epsilon}} \right)$ FO complexity.
\end{restatable}
As already mentioned, our method admits $\tilde{\mathcal{O}}\left(n/\epsilon^{1/3}+1/\sqrt{\epsilon}\right)$ first-order complexity, significantly improving upon the $\mathcal{O}(n/\epsilon)$ of $\mathrm{SGD}$ and the  $\mathcal{O}\left(n^2 \log(1/\epsilon)\right)$ of VR methods. The description of our method (Algorithm~\ref{alg:1}) as well as the main steps for establishing Theorem~\ref{thm:1} are presented in Section~\ref{section:convergence-results}.

In Theorem~\ref{t:side} we present the formal guarantees of the naive baseline result that is a \textit{sparse variant} of $\mathrm{SGD}$ for continual finite-sum minimization~\cref{d:instance-optimal} with $\mathcal{O}(1/\epsilon^2)$ FO complexity. 
\begin{restatable}{theorem}{static2}\label{t:side}
There exists a first-order method, $\mathrm{SGD}$-$\mathrm{sparse}$ (Algorithm~\ref{alg:sparseSGD}), for continual finite-sum minimization~\cref{d:instance-optimal} with
$\cO\left(\frac{G^{3}|\mathcal{D}|\log n}{\mu \epsilon^2} \right)$ FO complexity.
\end{restatable}
Algorithm~\ref{alg:sparseSGD} as well as the proof of Theorem~\ref{t:side} are presented in Appendix~\ref{app:lowaccuracy}.

\textbf{Lower Bounds for Natural First-Order Methods}
In view of the above, the question that naturally arises is whether there exists a method for continual finite-sum minimization~\cref{d:instance-optimal} with $O(n^{2-\alpha} \log (1/\epsilon))$ FO complexity. Unfortunately the latter goal cannot be met for a wide class of first-order methods, that we call \textit{natural}. The formal definition of \textit{natural first-order method} is presented in Section~\ref{s:natural}. To this end, we remark that methods such that  $\mathrm{GD}$, $\mathrm{SGD}$, $\mathrm{SVRG}$, $\mathrm{SAGA}$, $\mathrm{SARAH}$, $\mathrm{Katyusha}$ and $\mathrm{CSVRG}$ (Algorithm~\ref{alg:1}) lie in the above class.

\begin{restatable}{theorem}{imp}\label{c:imp}
For any $\alpha >0$, there is no natural first-order method for Problem~\cref{d:instance-optimal} with $\mathcal{O}\left(n^{2-\alpha} \log(1/\epsilon)\right)$ FO complexity. Moreover for any $\alpha <1/4$, there is no natural first-order method for continual finite-sum minimization~\cref{d:instance-optimal} with $\mathcal{O}\left(n/ \epsilon^\alpha\right)$ FO complexity.
\end{restatable}
Due to space limitations the proof of Theorem~\ref{c:imp} is presented in Appendix~\ref{app:lowerbound}.
\begin{remark}
The $\tilde{\mathcal{O}}(n/\epsilon^{1/3} +1/\sqrt{\epsilon})$ FO complexity of CSVRG is close to the $\Omega(n/\epsilon^{1/4})$ lower bound.
\end{remark}

\subsection{Natural First-Order methods}\label{s:natural}
In this section we provide the formal definition of natural first-order method for which our lower bound stated in Theorem~\ref{c:imp} applies. Before doing so we present the some definitions characterizing general first-order methods.

At an intuitive level a \textit{first-order method} for continual finite-sum minimization~\cref{d:instance-optimal} determines $\hat{x}_i \in \mathcal{D}$ by only requiring first-order access to the functions $f_1,\ldots,f_i$. Specifically at each stage $i \in [n]$, a first-order method makes queries of the form $\{\nabla f_j(x) \text{ for } j\leq i\}$ and uses this information to determine $\hat{x}_i \in \mathcal{D}$. The latter intuitive description is formalized in Definition~\ref{def:complex}.

\begin{definition}\label{def:complex}
Let $\cA$ be a first-order method for continual finite-sum minimization~\cref{d:instance-optimal}. At each stage $i \in [n]$,
\begin{itemize}
\item $\hat{x}_0 \in \mathcal{D}$ denotes the initial point of $\mathcal{A}$ and $\hat{x}_i \in \mathcal{D}$ denotes the output of $\mathcal{A}$ at stage $i \in [n]$.

\item $x_i^t \in \mathcal{D}$ denotes the intermediate point of $\mathcal{A}$ at round $t \geq 1$ of stage $i \in [n]$ and $T_i \in \mathbb{N}$ the number of iterations during stage $i \in [n]$.

\item At each round $t \in [T_i]$ of stage $i\in [n]$, $\mathcal{A}$ performs a set of first-order oracles denotes $Q_i^t$. 
\begin{itemize}
\item Each first-order oracle $q\in Q^t_i$ admits the form $q = (q_{\mathrm{value}},q_{\mathrm{index}})$ where $q_{\mathrm{value}} \in \mathcal{D}$ denotes the queried point and $q_{\mathrm{index}} \leq i$ denotes the index of the queried function.
\item $\mathcal{A}$ computes the gradients 
$\{\nabla f_{q_{\mathrm{index}}}(q_{\mathrm{value}}) \text{ for all }q \in Q_i^t\}$ and uses this information to determine $x_i^{t+1} \in \mathcal{D}$.
\end{itemize}
\end{itemize}
The FO complexity of $\mathcal{A}$ is thus $\sum_{i \in [n]}\sum_{t=1}^{T_i}|Q_i^t|$.
\end{definition}

\begin{example}
Gradient Descent at stage $i \in [n]$, sets  $~x_i^0 
\leftarrow \hat{x}_{i-1}$ and for each round $t \in [T_i]$ peforms $x_i^t \leftarrow \Pi_{\mathcal{D}}\left[ x_i^{t-1} - \gamma \sum_{j=i}^i f_j(x_i^{t-1})/i\right]$. Finally it sets $\hat{x}_i \leftarrow x_i^{T_i}$. Thus in terms of Definition~\ref{def:complex} the first-order oracles of $\mathrm{GD}$ at round $t \in T_i$ are $Q_i^t = \{ \left(x_i^{t-1},1\right),\ldots, \left(x_i^{t-1},i\right)\}$. 
\end{example}
In Definition~\ref{d:nat_method} we formally define the notion of a \textit{natural} first-order method. In simple terms a first-order method is called \textit{natural} if during each stage $i\in [n]$ it only computes FOs of the form $\nabla f_j(x)$ where $x \in \mathcal{D}$ is previously generated point and $j \leq i$. 

\begin{definition}\label{d:nat_method}
A first-order method $\mathcal{A}$ is called natural iff at each step $t \in [T_i]$ of stage $i\in [n]$,
\begin{itemize}
\item For any query $q \in Q_i^t$, $q_{\mathrm{index}} \leq i $ and $q_{\mathrm{value}} \in \mathrm{PP}_i^{t-1}$ where $\mathrm{PP}_i^{t-1} := \left(\cup_{m \leq i-1} \cup_{\tau \leq T_m} x_m^\tau\right) \cup \left(\cup_{\tau \leq t-1} x_i^\tau \right) \cup 
\hat{x}_0$ ($\mathrm{PP}_i^{t-1}$ stands for previous points)

\item $x_i^t \in \Pi_{\mathcal{D}}(S)$ where $S$ is the linear span of $\left( \cup_{q \in Q_i^t} \nabla f_{q_{\mathrm{index}}}(q_\mathrm{value})\right) \cup \mathrm{PP}_i^{t-1}$.

\item $\hat{x}_i \in \Pi_{\mathcal{D}}(S)$ where $S$ is the linear span of $\mathrm{PP}_i^{T_i}$.
\end{itemize}
\end{definition}

\begin{remark}
Definition~\ref{d:nat_method} also captures randomized first-order methods such as $\mathrm{SGD}$ by considering $Q_i^t$ being a random set. 
\end{remark}
Natural first-order methods is the straightforward extension of \textit{linear-span methods}~\cite{B15convex,N14} in the context of continual finite-sum minimization~\cref{d:instance-optimal}. Linear-span methods require that for any first-order oracle $\nabla f(x)$ computed at iteration $t$, $x \in \{x_0, x_1,\ldots,x_{t-1}\}$. Moreover $x_t$ is required to lie in the linear span of the previous points and the computed gradients. Linear-span methods capture all natural first-order optimization methods and have been used to establish the well-known $\Omega(1/\sqrt{\epsilon})$ and $\Omega \left (\sqrt{L/\mu}\log (1/\epsilon)\right)$ respectively for the convex and strongly convex case (see Section 3.5 in \cite{B15convex}).

\section{CSVRG and Convergence Results}\label{section:convergence-results}
In this section we present our first-order method $\mathrm{CSVRG}$ that is able to achieve the guarantees established in Theorem~\ref{thm:1}. $\mathrm{CSVRG}$ is formally described in Algorithm~\ref{alg:1}.

\begin{algorithm}[ht]
    \caption{$\mathrm{CSVRG}$}
    \label{alg:1}
\begin{algorithmic}[1]    

   \State  $\hat{x}_0 \in \mathcal{D}$, $\prev \leftarrow 0$, $update \leftarrow false$\;

   \State $\hat{x}_1 \leftarrow \mathrm{GradientDescent}(\hat{x}_0)$, $\redvargrad_1 \leftarrow \nabla f_1(\hat{x}_1)$\;
\smallskip
    \For {each stage $i = 2,\dots,n$}
         \smallskip
         \If{$i - \prev \geq \alpha \cdot i$}
             \smallskip
             \State $\redvargrad_{i-1} \leftarrow \frac{1}{i-1}\sum_{j=1}^{i-1}\nabla f_j(\hat{x}_{i-1})$ \Comment{Compute  $\nabla g_{i-1}(\hat{x}_{i-1})$ with $i-1$ $\mathrm{FO}$s}
             \smallskip
             \State $\prev \leftarrow i-1$
            \smallskip
            \State $update \leftarrow true$
            \smallskip
         \EndIf
         \smallskip
         \State $T_i \leftarrow \cO\left(\frac{L^2G}{\mu^{5/2}i\sqrt{\epsilon}} + \frac{L^2G^2\alpha^2}{\mu\epsilon} + \frac{L^2}{\mu^2}\right)$
         \Comment{Number of iterations of Algorithm~\ref{alg:frequent} at stage $i \in [n]$}
         \smallskip
        \State  $\hat{x}_i\leftarrow \mathrm{FUM}(\prev,\redvargrad_{i-1},T_i)$ \Comment{Selection of $\hat{x}_{i} \in \mathcal{D}$ by Algorithm~\ref{alg:frequent} } 
        \smallskip
        \If{$update$}
            \smallskip
            \State $\redvargrad_i \leftarrow \frac{1}{i}\sum_{j=1}^{i}\nabla f_j(\hat{x}_i)$ \Comment{Compute  $\nabla g_{i}(\hat{x}_{i})$ with $i$ $\mathrm{FO}$s}
            \smallskip
            \State $\prev \leftarrow i$
            \smallskip
            \State $update \leftarrow false$

        \Else
        \State    $\redvargrad_i \leftarrow \left(1-\frac{1}{i}\right)\redvargrad_{i-1} + \frac{1}{i}\nabla f_i(\hat{x}_{\prev})$\Comment{Update $\redvargrad_i$ with $1$ $\mathrm{FO}$}
    \EndIf
    \EndFor
    \end{algorithmic}
\end{algorithm}

\begin{algorithm}[ht]
    \caption{\textbf{F}requent \textbf{U}pdate \textbf{M}ethod (\textbf{$\mathrm{FUM}$})}
    \label{alg:frequent}
\begin{algorithmic}[1]

 \State    $\beta \leftarrow \frac{72 L^2}{\mu^2}$, $\cZ \leftarrow \frac{T_i(T_i-1)}{2}+(T_i+1)(\beta-1)$
 \smallskip 
    \State $x_i^0 \leftarrow \hat{x}_{i-1}$ \Comment{Initialization at $\hat{x}_{i-1} \in \mathcal{D}$}
    \smallskip
    \For{each round $t:= 1,\dots,T_i$}
        \smallskip
        \State Select $u_t \sim \mathrm{Unif}(1,\dots,i-1)$
        \smallskip
        \State $\nabla_i^t \leftarrow \left(1 - \frac{1}{i}\right)\left(\nabla f_{u_t}(x_i^t)-\nabla f_{u_t}(\hat{x}_{\prev})+\redvargrad_{i-1}\right) + \frac{1}{i}\nabla f_i(x_i^t)$\Comment{$3$ FOs}
        \smallskip
        \State $\gamma_t \leftarrow 4/(\mu(t+\beta))$ \Comment{Step-size selection}
        \smallskip 
        \State $x_i^{t+1} \leftarrow \Pi_\cD \left(x_i^t - \gamma_t \nabla_i^t\right)$\Comment{Update $x_i^t \in \mathcal{D}$}
        \smallskip
    \EndFor
    \State \textbf{Output:}~~ $\hat{x}_i \leftarrow \frac{1}{\cZ}\sum_{s=0}^{T_i-1}(s+\beta-1)x_i^{t+1}$\Comment{Final output}
    \end{algorithmic}
\end{algorithm}

We first remark that the computation of the output $\hat{x}_i \in \mathcal{D}$ for each stage $i \in [n]$, is performed at Step~$10$ of Algorithm~\ref{alg:1} by calling Algorithm~\ref{alg:frequent}. Then Algorithm~\ref{alg:frequent} initializes $x_i^0 := \hat{x}_{i-1}$ and performs $T_i$ stochastic gradient steps using the estimator
\[\nabla_i^t \leftarrow \left(1 - \frac{1}{i}\right)\left(\nabla f_{u_t}(x_i^t)-\nabla f_{\textcolor{blue}{u_t}}(\textcolor{red}{\hat{x}_{\prev}})+\redvargrad_{i-1}\right) + \frac{1}{i}\nabla f_i(x_i^t)\]
To this end we remark that in order to compute $\nabla_i^t$ Algorithm~\ref{alg:frequent} requires just $3$ FOs. Before explaining the specific selection of the above estimator we start by explaining the role of $\redvargrad_{i}$. 

\textbf{The role of $\redvargrad_{i}$}: In order to keep the variance of the estimator $\nabla_i^t$ low, we would ideally want $\redvargrad_{i-1} := \nabla g_{i-1}(\hat{x}_{i-1})$. The problem is that computing  $\nabla g_i(\hat{x}_i)$ requires $i$ FOs meaning that performing such a computation at each stage $i\in [n]$, would directly result in $\Omega(n^2)$ FOs. In order to overcome the latter challenge, a full gradient $\nabla g_i(x)$ is computed very sparsely by Algorithm~\ref{alg:1} across the $n$ stages. Specifically, 
\begin{itemize}
\item In case $i - \prev \geq \alpha \cdot i$ then $\redvargrad_{i} := \sum_{j=1}^{i}\nabla f_j(\hat{x}_{i})/i$~~~~~~~~~~~~~~~~~~~~~~~~~~~~~~~~ ~~~~~~($i$ FOs)
\item In case $i - \prev < \alpha \cdot i$ then $\redvargrad_{i} :=        \left(1-\frac{1}{i}\right)\redvargrad_{i-1} + \frac{1}{i}\nabla f_i(\hat{x}_{\prev})$~~~~~~~~~~~~~~~~~~(1 FO)
\end{itemize}
Thus the parameter $\alpha >0$ controls the number of times Algorithm~\ref{alg:1} reaches Steps~$5$ and~$12$ that at stage $i \in [n]$ require $i$ $\mathrm{FO}$s. The latter is formally stated and established in Lemma~\ref{l:sparse}.
\begin{restatable}{lemma}{sparse}\label{l:sparse}
Over a sequence of $n$ stages, Algorithm~\ref{alg:1} reaches Step~$5$ and~$12$, $\lceil\log n / \alpha \rceil$ times. 
\end{restatable}
Combining the latter with the fact that Algorithm~\ref{alg:frequent} requires $3T_i$ $\mathrm{FO}$s at each stage $i \in [n]$ and the fact that Step~$5$ and~$12$ require at most $n$ FOs, we get that
\begin{restatable}{corollary}{fos}\label{l:FOs}
Over a sequence of $n$ stages, Algorithm~\ref{alg:1} requires $3\sum_{i=1}^n T_i + 2n \lceil\log n/\alpha\rceil$ $\mathrm{FO}$s.
\end{restatable}
Up next we present Theorem~\ref{t:accuracy} establishing that for a specific selection of parameter $\alpha >0$, Algorithm~\ref{alg:1} guarantees that each $\hat{x}_i \in \mathcal{D}$ is an $\epsilon$-optimal point for the function $g_i(x)$.
\begin{restatable}{theorem}{accuracy}\label{t:accuracy}
Let a convex and compact set $\mathcal{D}$ and a sequence $\mu$-strongly convex functions $f_1,\ldots,f_n$ with $f_i:\mathcal{D}\mapsto \mathbb{R}$. Then Algorithm~\ref{alg:1}, with $T_i = 720GL^2/(\mu^{5/2}i\sqrt{\epsilon})+9L^{2/3}G^{2/3} / (\epsilon^{1/3}\mu)+864L^2/\mu^2$
and $\alpha = \mu\epsilon^{1/3} / (20G^{2/3}L^{2/3})$, guarantees
\[\bE\left[g_i(\hat{x}_i) \right] - g_i(x^\star_i) \leq \epsilon~~~\text{for each stage } i \in [n]\]
where $\hat{x}_i \in \mathcal{D}$ is the output of Algorithm~\ref{alg:frequent} at Step~$9$ of Algorithm~\ref{alg:1}.
\end{restatable}
The proof of Theorem~\ref{thm:1} directly follows by Theorem~\ref{t:accuracy} and Corollary~\ref{l:FOs}. For completeness the proof Theorem~\ref{thm:1} is presented in Appendix~\ref{app:const-approx}. In the rest of the section we present the key steps for proving Theorem~\ref{t:accuracy} (see Appendix~\ref{app:thm3}) for the full proof. 

We first explain the role of the gradient estimator $\nabla_i^t$ in Step~5 of Algorithm~\ref{alg:frequent}. This estimator $\nabla_i^t$ of Step~$5$ in Algorithm~\ref{alg:frequent} may seem unintuitive at first sight since it subtracts the term $\nabla f_{\textcolor{blue}{u_t}}(\textcolor{red}{\hat{x}_{\prev}})$. Interestingly enough the latter estimator-selection permits us to establish the following two key properties for $\nabla_i^t$.
\begin{restatable}[Unbiased]{lemma}{unbias}\label{l:unbias}
Let $\nabla_i^t$ the gradient estimator used in Step~$5$ of Algorithm~\ref{alg:frequent}. Then for all $t \in [T_i]$, $\mathbb{E}\left[ \nabla_i^t  \right] = \nabla g_i(x_i^t)$.
\end{restatable}

\begin{restatable}[Bounded Variance]{lemma}{boundvariance}\label{l:bounded_variance}
Let $\nabla_i^t$ the gradient estimator used in Step~$5$ of Algorithm~\ref{alg:frequent}. Then for all $t \in [T_i]$,
\[\bE\left[\norm{\nabla_i^t-\nabla g_i(x_t)}_2^2\right] \leq 8L^2\bE\left[\norm{x_i^t - x_i^\star}_2^2\right]+64\frac{L^2G^2}{\mu^2}\cdot \alpha^2 + 16L^2 \bE\left[\norm{x_{\prev}^\star - \hat{x}_{\prev}}^2_2\right]\]
where $\alpha > 0$ is  the parameter used at Step~$4$ of Algorithm~\ref{alg:1}.
\end{restatable}
The proof of Lemma~\ref{l:unbias} and Lemma~\ref{l:bounded_variance} are respectively presented in Appendix~\ref{app:bias} and Appendix~\ref{app:variance} and consist one of the main technical contributions of this work. In Section~\ref{section:estimator-analysis} we explain why the specific estimator-selection $\nabla_i^t$ is crucial for establishing both Lemma~\ref{l:unbias} and~\ref{l:bounded_variance}.  

In the rest of the section, we provide the main steps for establishing Theorem~\ref{t:accuracy}.  Let us first inductively assume that $\bE\left[ g_j(\hat{x}_j) - g_j(x_j^\star) \right] \leq \epsilon$ for all $j \leq i-1$. We use the latter to establish that $\bE\left[g_i(\hat{x}_i) - g_i(x_i^\star) \right] \leq \epsilon$.

Notice that Step~$4$ of Algorithm~\ref{alg:1} guarantees that $\prev \leq i-1$ at Step~$10$. Hence the induction hypothesis ensures that $\bE\left[ g_{\prev}(\hat{x}_{\prev}) - g_{\prev}(x_{\prev}^\star) \right] \leq \epsilon$ and thus by strong convexity $\bE\left[ \norm{\hat{x}_{\prev} - x_{\prev}^\star}^2 \right] \leq  2\epsilon/\mu$. Thus the bound in the variance of Lemma~\ref{l:bounded_variance},
\begin{equation}\label{eq:new_var}
\bE\left[\norm{\nabla_i^t-\nabla g_i(x_t)}_2^2\right] \leq \mathcal{O}\left(L^2\bE\left[\norm{x_i^t - x_i^\star}_2^2\right]+\frac{L^2G^2}{\mu^2}\cdot \alpha^2 + \frac{L^2}{\mu}\epsilon\right)
\end{equation}
Using the fact that $\bE\left[\nabla_i^t \right] = \nabla g_i(x_i^t)$ and the bound in the variance $\bE\left[\norm{\nabla_i^t-\nabla g_i(x_t)}_2^2 \right]$ of Equation~\ref{eq:new_var},  in Lemma~\ref{lemm:convergebound} of Appendix~\ref{app:convbound} we establish that
\[\bE\left[g_i(\hat{x}_i) \right] - g_i(x_i^\star) := \bE\left[g_i(x_i^{T_i}) \right] - g_i(x_i^\star)\leq  \mathcal{O}\left(\frac{L^4}{\mu^3}\frac{\norm{x_i^0-x_i^\star}_2^2}{T_i^2}+\frac{L^2 G^2}{\mu^3}\frac{\alpha^2}{ T_i}+\frac{L^2\epsilon}{\mu^2 T_i}\right). \]
Notice that at Step~$2$ of Algorithm~\ref{alg:frequent}, $x_i^0 := \hat{x}_{i-1} \in \mathcal{D}$ meaning that  
\begin{eqnarray*}
    \bE\left[g_i(\hat{x}_i) \right] - g_i(x_i^\star) &\leq&  \mathcal{O}\left(\frac{L^4}{\mu^3} \frac{\norm{\hat{x}_{i-1}-x_i^\star}_2^2}{T_i^2} +\frac{L^2G^2}{\mu^3}\frac{\alpha^2}{ T_i}+\frac{L^2\epsilon}{\mu^2 T_i}\right)\\
    &\leq& \mathcal{O}\left(\frac{L^4G^2}{\mu^5i^2 
    \cdot T_i^2}+\frac{L^4\epsilon}{\mu^4 T_i^2} +\frac{L^2G^2\alpha^2}{\mu^3 T_i}+\frac{L^2\epsilon}{\mu^2 T_i}\right)
\end{eqnarray*}
where the last inequality follows by $\norm{\hat{x}_{i-1} -x^\star_i}_2^2 \leq 2\norm{\hat{x}_{i-1} -x^\star_{i-1}}_2^2 + 2\norm{x^\star_{i-1} -x^\star_i}_2^2 $, $\norm{x_i^\star - x_{i-1}^\star} \leq \mathcal{O}(G^2/(\mu i))$ and $\norm{\hat{x}_{i-1} - x^\star_{i-1}} \leq \epsilon/\mu$. As a result, by taking $T_i = \mathcal{O}\left( \frac{L^2G}{\mu^{5/2}i\sqrt{\epsilon}} + \frac{L^2G^2\alpha^2}{ \mu^3\epsilon} + \frac{L^2}{\mu^2} \right)$ we get that $\bE\left[g_i(\hat{x}_i) \right] - g_i(x_i^\star) \leq  \epsilon$.

\section{Analyzing the Estimator \texorpdfstring{$\nabla_i^t$}{}}
\label{section:estimator-analysis}
In this section we provide the key steps for proving Lemma~\ref{l:unbias} and Lemma~\ref{l:bounded_variance} respectively.

\textbf{Unbias estimator:} The basic step for showing that 
$\nabla_i^t$ of Step~$5$ in Algorithm~\ref{alg:frequent} 
is an unbiased estimator, $\bE\left[\nabla_i^t \right] = \nabla g_i(x_i^t)$, is establishing the following property for $\redvargrad_{i-1}$ defined in Algorithm~\ref{alg:1}.

\begin{restatable}{lemma}{inductvar}\label{lemm:induction_var}
At Step~$10$ of Algorithm~\ref{alg:1}, it holds that $\redvargrad_{i-1} = \sum_{k=1}^{i-1} \nabla f_k(\hat{x}_{\prev})/(i-1)$
\end{restatable}

The proof of Lemma~\ref{lemm:induction_var} is presented in Appendix~\ref{app:bias} and admits an inductive proof on the steps of Algorithm~\ref{alg:1}. Once Lemma~\ref{lemm:induction_var} is established the fact that $\bE\left[\nabla_i^t\right] = \nabla f(x_i^t)$ easily follows by the selection of the estimator (see Step~$5$ of Algorithm~\ref{alg:frequent}). For the full proof of Lemma~\ref{l:unbias} we refer the reader to Appendix~\ref{app:bias}.

\textbf{Bounding the Variance of \texorpdfstring{$\nabla_i^t$}{}:}
The first step towards establishing Lemma~\ref{l:bounded_variance} is presented in Lemma~\ref{lemm:variance} providing a first bound on the variance $\bE\left[\norm{\nabla_i^t-\nabla g_i(x_i^t)}_2^2\right]$. The proof of Lemma~\ref{lemm:variance} uses ideas derived from the analysis of VR methods and is presented in Appendix~\ref{app:variance}.
    \begin{restatable}{lemma}{variance}
        \label{lemm:variance}
For all rounds $t \in T_i$ of stage $i \in [n]$,
\begin{equation}\label{eq:variance}
            \bE\left[\norm{\nabla_i^t-\nabla g_i(x_i^t)}_2^2\right] \leq 8L^2\bE\left[\norm{x_i^t - x_i^\star}_2^2\right]+8L^2\bE\left[\norm{x_i^\star - \hat{x}_{\prev}}_2^2\right]
\end{equation}
where $\prev$ is defined in Algorithm~\ref{alg:1}.
\end{restatable}
The next step for establishing Lemma~\ref{l:bounded_variance} is handling
the term $\bE\left[\norm{x_i^\star - \hat{x}_{\prev}}_2^2\right]$ of Equation~\ref{eq:variance}. 
In order to do the latter in we exploit the strong convexity assumption so as to establish Lemma~\ref{lemm:distbound2} the proof of which lies in Appendix~\ref{app:cons-optimum}.
\begin{restatable}{lemma}{distancebound}
    \label{lemm:distbound2}
    Let $\hat{x}_j \in \mathcal{D}$ the output of Algorithm~\ref{alg:1} for stage $j\in [n]$. Then for all $i \in \{j+1,\dots,n\}$,
    \[\norm{\hat{x}_j-x_{i}^\star}_2^2\leq \frac{8}{\mu^2}\left(\frac{G(i-j)}{i+j}\right)^2+2\norm{\hat{x}_j - x^\star_j}^2_2\]
\end{restatable}
Now by applying Lemma~\ref{lemm:distbound2} for $j := \prev$ we get that
 \begin{eqnarray*}
            \norm{x_i^\star - \hat{x}_{\prev}}_2^2 &\leq& \frac{8}{\mu^2}\left(\frac{G(i-\prev)}{2\prev+(i-\prev)}\right)^2 +2\norm{\hat{x}_{\prev} - x^\star_{\prev}}^2\\
            &\leq& \frac{8}{\mu^2}\left(\frac{G(i-\prev)}{i}\right)^2 +2\norm{\hat{x}_{\prev} - x^\star_{\prev}}^2
\end{eqnarray*}
Step~$4$ and~$6$ of Algorithm~\ref{alg:1} ensure that at Step~$10$ of Algorithm~\ref{alg:1} (when Algorithm~\ref{alg:frequent} is called), $i-\prev \leq \alpha \cdot i$. Thus,
\begin{equation*}\label{eq:12}
\norm{x_i^\star - \hat{x}_{\prev}}_2^2 \leq \frac{8}{\mu^2}\left(\frac{G(i-\prev)}{i}\right)^2 +2\norm{\hat{x}_{\prev} - x^\star_{\prev}}^2 \leq \frac{32 G^2}{\mu^2}\alpha^2 +2\norm{\hat{x}_{\prev} - x^\star_{\prev}}^2
\end{equation*}
and thus $\bE\left[\norm{x_i^\star - \hat{x}_{\prev}}_2^2\right] \leq \frac{32 G^2}{\mu^2}\alpha^2 +2\bE\left[\norm{\hat{x}_{\prev} - x^\star_{\prev}}^2\right]$.
Combining the latter with Equation~\ref{eq:variance} we overall get,
\[\bE\left[\norm{\nabla_i^t-\nabla g_i(x_i^t)}_2^2\right] \leq \mathcal{O}\left(L^2\bE\left[\norm{x_i^t - x_i^\star}_2^2\right]+\frac{L^2G^2}{\mu^2}\alpha^2 +L^2\bE\left[\norm{\hat{x}_{\prev} - x^\star_{\prev}}^2\right]\right)\]
\vspace{-4mm}

\section{Experiments}
\label{section:experiments}
\begin{table}[ht]
\centering
\begin{tabular}{ccc}
\includegraphics[width=0.31\textwidth]{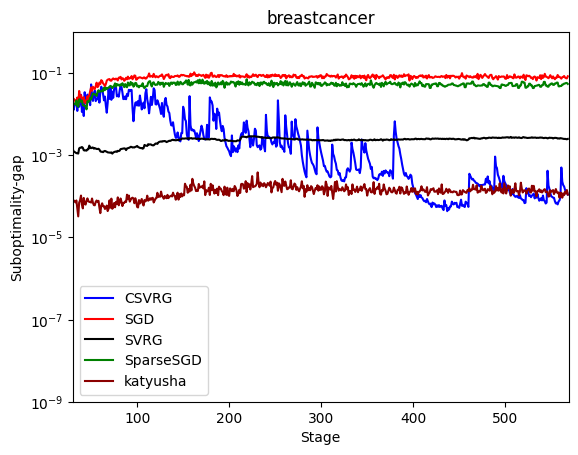} & \includegraphics[width=0.31\textwidth]{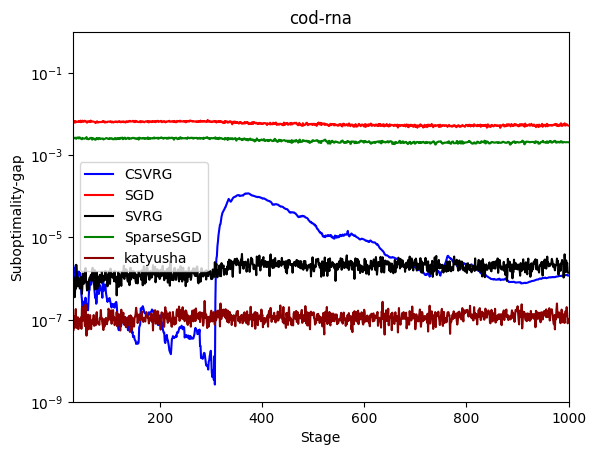}& \includegraphics[width=0.31\textwidth]{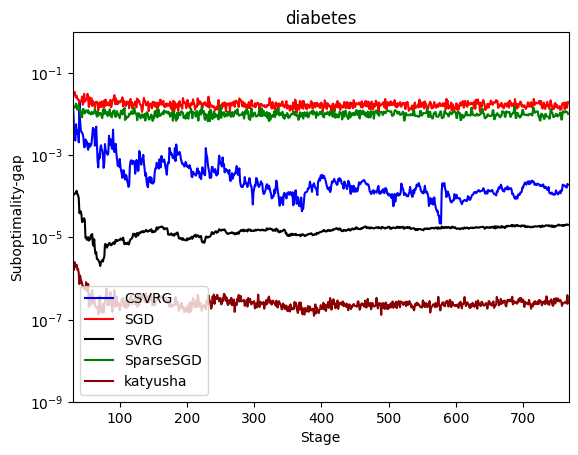} \\
\includegraphics[width=0.31\textwidth]{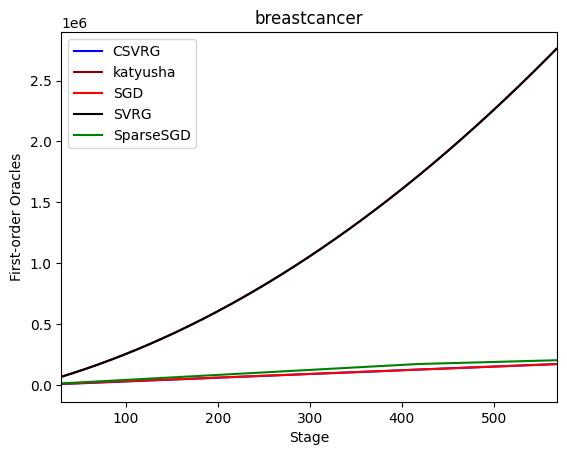} & \includegraphics[width=0.31\textwidth]{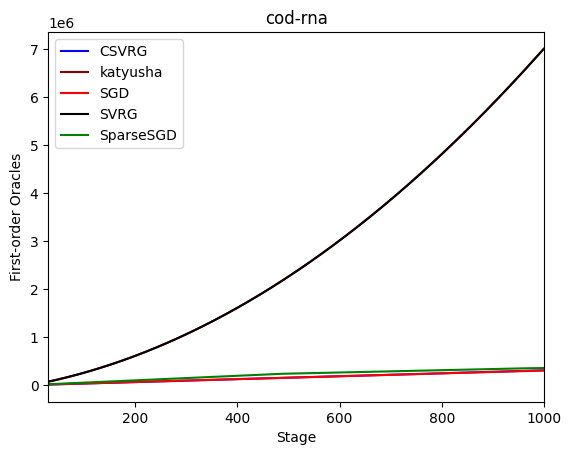} & \includegraphics[width=0.31\textwidth]{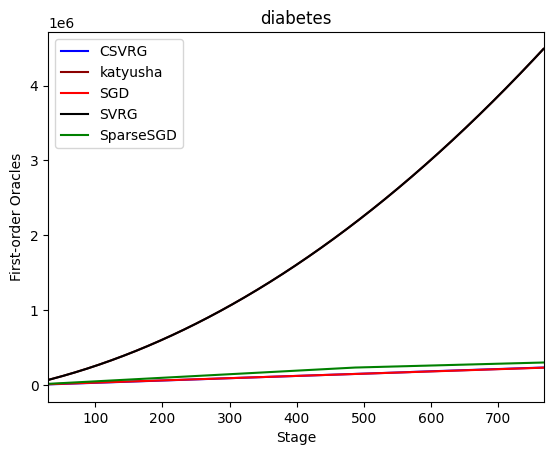}
\end{tabular}
\caption{Optimality gap as the stages progress on a ridge regression problem (averaged over $10$ independent runs). $\texttt{CSVRG}$ performs the exact same number of FOs with $\texttt{SGD}$ and slightly less than $\texttt{SGD}$-$\texttt{sparse}$. $\texttt{Katyusha}$ and $\texttt{SVRG}$ perform the exact same number of FOs. $\texttt{CSVRG}/\texttt{SGD}/\texttt{SGD}$-$\texttt{sparse}$ perform roughly $4\%$ of the FOs of $\texttt{Katyusha}/\texttt{SVRG}$.
\label{tab:experiments}}   
\end{table} 
\vspace{-3mm}
We experimentally evaluate the methods (\texttt{SGD},\texttt{SGD}-$\mathrm{sparse}$, \texttt{Katyusha}, \texttt{SVRG} and \texttt{CSVRG}) on a ridge regression task. Given some dataset $(a_i, b_i)_{i=1}^{n} \in \mathbb{R}^{d}\times \mathbb{R}$, at each stage $i\in [n]$ we consider the the finite sum objective $g_i(x) := \sum_{j=1}^{i} (a_j^\top x - b_j)^2/i + \lambda \norm{x}_2^2$ with $\lambda = 10^{-3}$. We choose the latter setting so as to be able to compute the exact optimal solution at each stage $i \in [n]$. 

For our experiments we use the datasets found in the LIBSVM package\cite{chang2011libsvm} for which we report our results. At each stage $i \in [n]$, we reveal a new data point $(a_i, b_i)$. In all of our experiments we run \texttt{CSVRG} with $\alpha =0.3$ and $T_i = 100$. The inner iterations of $\texttt{SGD}$ and $\texttt{SGD}$-$\texttt{sparse}$ are appropriately selected so that their overall FO calls match the FO calls of $\texttt{CSVRG}$. At each stage $i\in [n]$ we run \texttt{Katyusha}~\cite{allen2017katyusha} and \texttt{SVRG}~\cite{johnson2013accelerating} on the prefix-sum function $g_i(x)$ with $10$ outer iterations and $100$ inner iterations.   

In Table~\ref{tab:experiments} we present the error achieved by each method at each stage $i\in [n]$ and the overall number of FOs until the respective stage. As our experimental evaluations reveal, \texttt{CSVRG} nicely interpolates between $\texttt{SGD}$/$\texttt{SGD}$-$\mathrm{sparse}$ and $\mathrm{Katyusha}$/$\texttt{SVRG}$. More precisely, \texttt{CSVRG} achieves a significantly smaller suboptimality gap than $\texttt{SGD}$/$\texttt{SGD}$-$\mathrm{sparse}$ with the same number of FOs while it achieves a comparable one with $\texttt{Katyusha}$/$\texttt{SVRG}$ with way fewer FOs. In Appendix~\ref{app:expdetails} we present further experimental evaluations as well as the exact parameters used for each method.

\section{Conclusions}
\label{section:conclusions}
In this work we introduce \textit{continual finite-sum optimization}, a tweak of standard finite-sum minimization, that given a sequence of functions asks for a sequence of $\epsilon$-accurate solutions for the respective prefix-sum function. For the strongly convex case we propose a first-order method with $\mathcal{\tilde{O}}(n/\epsilon^{1/3})$ FO complexity. This significantly improves upon the $\mathcal{O}(n/\epsilon)$ FOs that $\mathrm{StochasticGradient Descent}$ requires and upon the $\mathcal{O}(n^2\log(1/\epsilon))$ FOs that VR methods require. We additionally prove that no \textit{natural method} can achieve $\mathcal{O}\left(n^{2-\alpha} \log (1/\epsilon)\right)$ FO complexity for any $\alpha >0$ and $\mathcal{O}\left(n /\epsilon^\alpha\right)$ FO complexity for $\alpha < 1/4$.  

\section{Acknowledgements}
Authors acknowledge the constructive feedback of reviewers and the work of ICLR'24 program and area chairs.  This work was supported by Hasler Foundation Program: Hasler Responsible
AI (project number 21043), by  the Swiss National Science Foundation (SNSF) under grant number 200021\_20501, by the Army Research Office (Grant Number W911NF-24-1-0048) and  by Innosuisse (contract agreement 100.960 IP-ICT).

\textbf{Limitations:} The theoretical guarantees of our proposed method are restricted to the strongly convex case. Extending our results beyond the strong convexity assumptions is a very promising research direction.

\textbf{Reproducibility Statement}
In the appendix we have included the formal proofs of all the theorems and lemmas provided in the main part of the paper. In order to facilitate the reproducibility of our theoretical results, for each theorem we have created a separate self-contained section presenting its proof. Concerning the experimental evaluations of our work, we provide the code used in our experiments as well as the selected parameters in each of the presented methods.

\textbf{Ethics Statement} The authors acknowledge that they have read and adhere to the ICLR Code of Ethics.

\bibliographystyle{plainnat}
\bibliography{ref.bib}

\begin{thebibliography}{63}
\providecommand{\natexlab}[1]{#1}
\providecommand{\url}[1]{\texttt{#1}}
\expandafter\ifx\csname urlstyle\endcsname\relax
  \providecommand{\doi}[1]{doi: #1}\else
  \providecommand{\doi}{doi: \begingroup \urlstyle{rm}\Url}\fi

\bibitem[Allen-Zhu(2017)]{allen2017katyusha}
Zeyuan Allen-Zhu.
\newblock Katyusha: The first direct acceleration of stochastic gradient methods.
\newblock In \emph{Proceedings of the 49th Annual ACM SIGACT Symposium on Theory of Computing}, pages 1200--1205, 2017.

\bibitem[Allen{-}Zhu(2018{\natexlab{a}})]{AZ18a}
Zeyuan Allen{-}Zhu.
\newblock Natasha 2: Faster non-convex optimization than {SGD}.
\newblock In Samy Bengio, Hanna~M. Wallach, Hugo Larochelle, Kristen Grauman, Nicol{\`{o}} Cesa{-}Bianchi, and Roman Garnett, editors, \emph{Advances in Neural Information Processing Systems 31: Annual Conference on Neural Information Processing Systems 2018, NeurIPS 2018, December 3-8, 2018, Montr{\'{e}}al, Canada}, pages 2680--2691, 2018{\natexlab{a}}.

\bibitem[Allen{-}Zhu(2018{\natexlab{b}})]{KatyushaXzhu}
Zeyuan Allen{-}Zhu.
\newblock Katyusha {X:} practical momentum method for stochastic sum-of-nonconvex optimization.
\newblock \emph{CoRR}, abs/1802.03866, 2018{\natexlab{b}}.
\newblock URL \url{http://arxiv.org/abs/1802.03866}.

\bibitem[Allen-Zhu and Yuan(2016)]{allen2016improved}
Zeyuan Allen-Zhu and Yang Yuan.
\newblock Improved svrg for non-strongly-convex or sum-of-non-convex objectives.
\newblock In \emph{International conference on machine learning}, pages 1080--1089. PMLR, 2016.

\bibitem[Blatt et~al.(2007)Blatt, Hero, and Gauchman]{blatt2007convergent}
Doron Blatt, Alfred~O Hero, and Hillel Gauchman.
\newblock A convergent incremental gradient method with a constant step size.
\newblock \emph{SIAM Journal on Optimization}, 18\penalty0 (1):\penalty0 29--51, 2007.

\bibitem[Bottou and Bousquet(2007)]{BB07}
L{\'{e}}on Bottou and Olivier Bousquet.
\newblock The tradeoffs of large scale learning.
\newblock In \emph{Advances in Neural Information Processing Systems 20, Proceedings of the Twenty-First Annual Conference on Neural Information Processing Systems}, pages 161--168. Curran Associates, Inc., 2007.

\bibitem[Boyd and Vandenberghe(2004)]{boyd2004convex}
Stephen~P Boyd and Lieven Vandenberghe.
\newblock \emph{Convex optimization}.
\newblock Cambridge university press, 2004.

\bibitem[Bubeck(2015)]{B15convex}
Sébastien Bubeck.
\newblock Convex optimization: Algorithms and complexity, 2015.

\bibitem[Castro et~al.(2018)Castro, Mar{\'\i}n-Jim{\'e}nez, Guil, Schmid, and Alahari]{castro2018end}
Francisco~M Castro, Manuel~J Mar{\'\i}n-Jim{\'e}nez, Nicol{\'a}s Guil, Cordelia Schmid, and Karteek Alahari.
\newblock End-to-end incremental learning.
\newblock In \emph{Proceedings of the European conference on computer vision (ECCV)}, pages 233--248, 2018.

\bibitem[Chang and Lin(2011)]{chang2011libsvm}
Chih-Chung Chang and Chih-Jen Lin.
\newblock Libsvm: a library for support vector machines.
\newblock \emph{ACM transactions on intelligent systems and technology (TIST)}, 2\penalty0 (3):\penalty0 1--27, 2011.

\bibitem[Defazio et~al.(2014)Defazio, Bach, and Lacoste{-}Julien]{DBL14}
Aaron Defazio, Francis~R. Bach, and Simon Lacoste{-}Julien.
\newblock {SAGA:} {A} fast incremental gradient method with support for non-strongly convex composite objectives.
\newblock In Zoubin Ghahramani, Max Welling, Corinna Cortes, Neil~D. Lawrence, and Kilian~Q. Weinberger, editors, \emph{Advances in Neural Information Processing Systems 27: Annual Conference on Neural Information Processing Systems 2014, December 8-13 2014, Montreal, Quebec, Canada}, pages 1646--1654, 2014.

\bibitem[Delyon and Juditsky(1993)]{delyon1993accelerated}
Bernard Delyon and Anatoli Juditsky.
\newblock Accelerated stochastic approximation.
\newblock \emph{SIAM Journal on Optimization}, 3\penalty0 (4):\penalty0 868--881, 1993.

\bibitem[Dubois{-}Taine et~al.(2022)Dubois{-}Taine, Vaswani, Babanezhad, Schmidt, and Lacoste{-}Julien]{DVBS22}
Benjamin Dubois{-}Taine, Sharan Vaswani, Reza Babanezhad, Mark Schmidt, and Simon Lacoste{-}Julien.
\newblock {SVRG} meets adagrad: painless variance reduction.
\newblock \emph{Mach. Learn.}, 111\penalty0 (12):\penalty0 4359--4409, 2022.

\bibitem[Dubois-Taine et~al.(2022)Dubois-Taine, Vaswani, Babanezhad, Schmidt, and Lacoste-Julien]{dubois2022svrg}
Benjamin Dubois-Taine, Sharan Vaswani, Reza Babanezhad, Mark Schmidt, and Simon Lacoste-Julien.
\newblock Svrg meets adagrad: painless variance reduction.
\newblock \emph{Machine Learning}, pages 1--51, 2022.

\bibitem[Fang et~al.(2018{\natexlab{a}})Fang, Li, Lin, and Zhang]{FLLZ18}
Cong Fang, Chris~Junchi Li, Zhouchen Lin, and Tong Zhang.
\newblock {SPIDER:} near-optimal non-convex optimization via stochastic path-integrated differential estimator.
\newblock In Samy Bengio, Hanna~M. Wallach, Hugo Larochelle, Kristen Grauman, Nicol{\`{o}} Cesa{-}Bianchi, and Roman Garnett, editors, \emph{Advances in Neural Information Processing Systems 31: Annual Conference on Neural Information Processing Systems 2018, NeurIPS 2018, December 3-8, 2018, Montr{\'{e}}al, Canada}, pages 687--697, 2018{\natexlab{a}}.

\bibitem[Fang et~al.(2018{\natexlab{b}})Fang, Li, Lin, and Zhang]{fang2018spider}
Cong Fang, Chris~Junchi Li, Zhouchen Lin, and Tong Zhang.
\newblock Spider: Near-optimal non-convex optimization via stochastic path-integrated differential estimator.
\newblock \emph{Advances in Neural Information Processing Systems}, 31, 2018{\natexlab{b}}.

\bibitem[Goodfellow et~al.(2014)Goodfellow, Mirza, Da, Courville, and Bengio]{GMDCB13}
Ian~J. Goodfellow, Mehdi Mirza, Xia Da, Aaron~C. Courville, and Yoshua Bengio.
\newblock An empirical investigation of catastrophic forgeting in gradient-based neural networks.
\newblock In Yoshua Bengio and Yann LeCun, editors, \emph{2nd International Conference on Learning Representations, {ICLR} 2014, Banff, AB, Canada, April 14-16, 2014, Conference Track Proceedings}, 2014.

\bibitem[Hazan(2023)]{H16}
Elad Hazan.
\newblock Introduction to online convex optimization, 2023.

\bibitem[Hersche et~al.(2022)Hersche, Karunaratne, Cherubini, Benini, Sebastian, and Rahimi]{H22}
Michael Hersche, Geethan Karunaratne, Giovanni Cherubini, Luca Benini, Abu Sebastian, and Abbas Rahimi.
\newblock Constrained few-shot class-incremental learning.
\newblock In \emph{Proceedings of the IEEE/CVF Conference on Computer Vision and Pattern Recognition}, pages 9057--9067, 2022.

\bibitem[Johnson and Zhang(2013)]{johnson2013accelerating}
Rie Johnson and Tong Zhang.
\newblock Accelerating stochastic gradient descent using predictive variance reduction.
\newblock \emph{Advances in neural information processing systems}, 26, 2013.

\bibitem[Kavis et~al.(2022)Kavis, Skoulakis, Antonakopoulos, Dadi, and Cevher]{KSADC22}
Ali Kavis, Stratis Skoulakis, Kimon Antonakopoulos, Leello~Tadesse Dadi, and Volkan Cevher.
\newblock Adaptive stochastic variance reduction for non-convex finite-sum minimization.
\newblock In \emph{NeurIPS}, 2022.

\bibitem[Kesten(1958)]{kesten1958accelerated}
Harry Kesten.
\newblock Accelerated stochastic approximation.
\newblock \emph{The Annals of Mathematical Statistics}, pages 41--59, 1958.

\bibitem[Kiefer and Wolfowitz(1952)]{kiefer1952stochastic}
Jack Kiefer and Jacob Wolfowitz.
\newblock Stochastic estimation of the maximum of a regression function.
\newblock \emph{The Annals of Mathematical Statistics}, pages 462--466, 1952.

\bibitem[Kirkpatrick et~al.(2017)Kirkpatrick, Pascanu, Rabinowitz, Veness, Desjardins, Rusu, Milan, Quan, Ramalho, Grabska-Barwinska, et~al.]{kirkpatrick2017overcoming}
James Kirkpatrick, Razvan Pascanu, Neil Rabinowitz, Joel Veness, Guillaume Desjardins, Andrei~A Rusu, Kieran Milan, John Quan, Tiago Ramalho, Agnieszka Grabska-Barwinska, et~al.
\newblock Overcoming catastrophic forgetting in neural networks.
\newblock \emph{Proceedings of the national academy of sciences}, 114\penalty0 (13):\penalty0 3521--3526, 2017.

\bibitem[Kowshik et~al.(2021{\natexlab{a}})Kowshik, Nagaraj, Jain, and Netrapalli]{kowshik2021near}
Suhas Kowshik, Dheeraj Nagaraj, Prateek Jain, and Praneeth Netrapalli.
\newblock Near-optimal offline and streaming algorithms for learning non-linear dynamical systems.
\newblock \emph{Advances in Neural Information Processing Systems}, 34:\penalty0 8518--8531, 2021{\natexlab{a}}.

\bibitem[Kowshik et~al.(2021{\natexlab{b}})Kowshik, Nagaraj, Jain, and Netrapalli]{kowshik2021streaming}
Suhas Kowshik, Dheeraj Nagaraj, Prateek Jain, and Praneeth Netrapalli.
\newblock Streaming linear system identification with reverse experience replay.
\newblock \emph{Advances in Neural Information Processing Systems}, 34:\penalty0 30140--30152, 2021{\natexlab{b}}.

\bibitem[Lan and Zhou(2018)]{lan2018optimal}
Guanghui Lan and Yi~Zhou.
\newblock An optimal randomized incremental gradient method.
\newblock \emph{Mathematical programming}, 171:\penalty0 167--215, 2018.

\bibitem[Lan et~al.(2019{\natexlab{a}})Lan, Li, and Zhou]{LZ19}
Guanghui Lan, Zhize Li, and Yi~Zhou.
\newblock A unified variance-reduced accelerated gradient method for convex optimization.
\newblock In Hanna~M. Wallach, Hugo Larochelle, Alina Beygelzimer, Florence d'Alch{\'{e}}{-}Buc, Emily~B. Fox, and Roman Garnett, editors, \emph{Advances in Neural Information Processing Systems 32: Annual Conference on Neural Information Processing Systems 2019, NeurIPS 2019, December 8-14, 2019, Vancouver, BC, Canada}, pages 10462--10472, 2019{\natexlab{a}}.

\bibitem[Lan et~al.(2019{\natexlab{b}})Lan, Li, and Zhou]{lan2019unified}
Guanghui Lan, Zhize Li, and Yi~Zhou.
\newblock A unified variance-reduced accelerated gradient method for convex optimization.
\newblock \emph{Advances in Neural Information Processing Systems}, 32, 2019{\natexlab{b}}.

\bibitem[Li and Li(2018)]{LL18}
Zhize Li and Jian Li.
\newblock A simple proximal stochastic gradient method for nonsmooth nonconvex optimization.
\newblock In Samy Bengio, Hanna~M. Wallach, Hugo Larochelle, Kristen Grauman, Nicol{\`{o}} Cesa{-}Bianchi, and Roman Garnett, editors, \emph{Advances in Neural Information Processing Systems 31: Annual Conference on Neural Information Processing Systems 2018, NeurIPS 2018, December 3-8, 2018, Montr{\'{e}}al, Canada}, pages 5569--5579, 2018.

\bibitem[Li and Richt{\'{a}}rik(2021)]{LP21}
Zhize Li and Peter Richt{\'{a}}rik.
\newblock Zerosarah: Efficient nonconvex finite-sum optimization with zero full gradient computation.
\newblock \emph{CoRR}, abs/2103.01447, 2021.

\bibitem[Li et~al.(2021)Li, Bao, Zhang, and Richt{\'{a}}rik]{LR21}
Zhize Li, Hongyan Bao, Xiangliang Zhang, and Peter Richt{\'{a}}rik.
\newblock {PAGE:} {A} simple and optimal probabilistic gradient estimator for nonconvex optimization.
\newblock In Marina Meila and Tong Zhang, editors, \emph{Proceedings of the 38th International Conference on Machine Learning, {ICML} 2021, 18-24 July 2021, Virtual Event}, volume 139 of \emph{Proceedings of Machine Learning Research}, pages 6286--6295. {PMLR}, 2021.

\bibitem[Lin et~al.(2018)Lin, Mairal, and Harchaoui]{lin2018catalyst}
Hongzhou Lin, Julien Mairal, and Zaid Harchaoui.
\newblock Catalyst acceleration for first-order convex optimization: from theory to practice.
\newblock \emph{Journal of Machine Learning Research}, 18\penalty0 (1):\penalty0 7854--7907, 2018.

\bibitem[Lin et~al.(2015)Lin, Lu, and Xiao]{lin2015accelerated}
Qihang Lin, Zhaosong Lu, and Lin Xiao.
\newblock An accelerated randomized proximal coordinate gradient method and its application to regularized empirical risk minimization.
\newblock \emph{SIAM Journal on Optimization}, 25\penalty0 (4):\penalty0 2244--2273, 2015.

\bibitem[Lopez-Paz and Ranzato(2017)]{lopez2017gradient}
David Lopez-Paz and Marc'Aurelio Ranzato.
\newblock Gradient episodic memory for continual learning.
\newblock \emph{Advances in neural information processing systems}, 30, 2017.

\bibitem[McCloskey and Cohen(1989)]{mccloskey1989catastrophic}
Michael McCloskey and Neal~J Cohen.
\newblock Catastrophic interference in connectionist networks: The sequential learning problem.
\newblock In \emph{Psychology of learning and motivation}, volume~24, pages 109--165. Elsevier, 1989.

\bibitem[Mermillod et~al.(2013)Mermillod, Bugaiska, and Bonin]{mermillod2013stability}
Martial Mermillod, Aur{\'e}lia Bugaiska, and Patrick Bonin.
\newblock The stability-plasticity dilemma: Investigating the continuum from catastrophic forgetting to age-limited learning effects, 2013.

\bibitem[Mokhtari et~al.(2017)Mokhtari, G{\"{u}}rb{\"{u}}zbalaban, and Ribeiro]{MGR17}
Aryan Mokhtari, Mert G{\"{u}}rb{\"{u}}zbalaban, and Alejandro Ribeiro.
\newblock A double incremental aggregated gradient method with linear convergence rate for large-scale optimization.
\newblock In \emph{2017 {IEEE} International Conference on Acoustics, Speech and Signal Processing, {ICASSP} 2017, New Orleans, LA, USA, March 5-9, 2017}, pages 4696--4700. {IEEE}, 2017.

\bibitem[Nagaraj et~al.(2020)Nagaraj, Wu, Bresler, Jain, and Netrapalli]{nagaraj2020least}
Dheeraj Nagaraj, Xian Wu, Guy Bresler, Prateek Jain, and Praneeth Netrapalli.
\newblock Least squares regression with markovian data: Fundamental limits and algorithms.
\newblock \emph{Advances in neural information processing systems}, 33:\penalty0 16666--16676, 2020.

\bibitem[Nesterov(2014)]{N14}
Yurii Nesterov.
\newblock \emph{Introductory Lectures on Convex Optimization: A Basic Course}.
\newblock Springer Publishing Company, Incorporated, 1 edition, 2014.
\newblock ISBN 1461346916.

\bibitem[Nguyen et~al.(2017{\natexlab{a}})Nguyen, Li, Bui, and Turner]{nguyen2017variational}
Cuong~V Nguyen, Yingzhen Li, Thang~D Bui, and Richard~E Turner.
\newblock Variational continual learning.
\newblock \emph{arXiv preprint arXiv:1710.10628}, 2017{\natexlab{a}}.

\bibitem[Nguyen et~al.(2017{\natexlab{b}})Nguyen, Liu, Scheinberg, and Tak{\'{a}}c]{NLST17}
Lam~M. Nguyen, Jie Liu, Katya Scheinberg, and Martin Tak{\'{a}}c.
\newblock {SARAH:} {A} novel method for machine learning problems using stochastic recursive gradient.
\newblock In Doina Precup and Yee~Whye Teh, editors, \emph{Proceedings of the 34th International Conference on Machine Learning, {ICML} 2017, Sydney, NSW, Australia, 6-11 August 2017}, volume~70 of \emph{Proceedings of Machine Learning Research}, pages 2613--2621. {PMLR}, 2017{\natexlab{b}}.

\bibitem[Pham et~al.(2020)Pham, Nguyen, Phan, and Tran{-}Dinh]{PNPT20}
Nhan~H. Pham, Lam~M. Nguyen, Dzung~T. Phan, and Quoc Tran{-}Dinh.
\newblock Proxsarah: An efficient algorithmic framework for stochastic composite nonconvex optimization.
\newblock \emph{J. Mach. Learn. Res.}, 21:\penalty0 110:1--110:48, 2020.

\bibitem[Reddi et~al.(2016)Reddi, Hefny, Sra, P{\'{o}}czos, and Smola]{RHSPS16}
Sashank~J. Reddi, Ahmed Hefny, Suvrit Sra, Barnab{\'{a}}s P{\'{o}}czos, and Alexander~J. Smola.
\newblock Stochastic variance reduction for nonconvex optimization.
\newblock In Maria{-}Florina Balcan and Kilian~Q. Weinberger, editors, \emph{Proceedings of the 33nd International Conference on Machine Learning, {ICML} 2016, New York City, NY, USA, June 19-24, 2016}, volume~48 of \emph{{JMLR} Workshop and Conference Proceedings}, pages 314--323. JMLR.org, 2016.

\bibitem[Robbins and Monro(1951)]{robbins1951stochastic}
Herbert Robbins and Sutton Monro.
\newblock A stochastic approximation method.
\newblock \emph{The annals of mathematical statistics}, pages 400--407, 1951.

\bibitem[Rosenfeld and Tsotsos(2018)]{RT18incremental}
Amir Rosenfeld and John~K Tsotsos.
\newblock Incremental learning through deep adaptation.
\newblock \emph{IEEE transactions on pattern analysis and machine intelligence}, 42\penalty0 (3):\penalty0 651--663, 2018.

\bibitem[Roux et~al.(2012{\natexlab{a}})Roux, Schmidt, and Bach]{roux2012stochastic}
Nicolas Roux, Mark Schmidt, and Francis Bach.
\newblock A stochastic gradient method with an exponential convergence \_rate for finite training sets.
\newblock \emph{Advances in neural information processing systems}, 25, 2012{\natexlab{a}}.

\bibitem[Roux et~al.(2012{\natexlab{b}})Roux, Schmidt, and Bach]{RSB12}
Nicolas~Le Roux, Mark Schmidt, and Francis~R. Bach.
\newblock A stochastic gradient method with an exponential convergence rate for finite training sets.
\newblock In Peter~L. Bartlett, Fernando C.~N. Pereira, Christopher J.~C. Burges, L{\'{e}}on Bottou, and Kilian~Q. Weinberger, editors, \emph{Advances in Neural Information Processing Systems 25: 26th Annual Conference on Neural Information Processing Systems 2012. Proceedings of a meeting held December 3-6, 2012, Lake Tahoe, Nevada, United States}, pages 2672--2680, 2012{\natexlab{b}}.

\bibitem[Ruvolo and Eaton(2013)]{ruvolo2013ella}
Paul Ruvolo and Eric Eaton.
\newblock Ella: An efficient lifelong learning algorithm.
\newblock In \emph{International conference on machine learning}, pages 507--515. PMLR, 2013.

\bibitem[Shalev-Shwartz and Zhang(2013{\natexlab{a}})]{shalev2013accelerated}
Shai Shalev-Shwartz and Tong Zhang.
\newblock Accelerated mini-batch stochastic dual coordinate ascent.
\newblock \emph{Advances in Neural Information Processing Systems}, 26, 2013{\natexlab{a}}.

\bibitem[Shalev-Shwartz and Zhang(2013{\natexlab{b}})]{shalev2013stochastic}
Shai Shalev-Shwartz and Tong Zhang.
\newblock Stochastic dual coordinate ascent methods for regularized loss minimization.
\newblock \emph{Journal of Machine Learning Research}, 14\penalty0 (1), 2013{\natexlab{b}}.

\bibitem[Shalev-Shwartz et~al.(2010)Shalev-Shwartz, Shamir, Srebro, and Sridharan]{SS10}
Shai Shalev-Shwartz, Ohad Shamir, Nathan Srebro, and Karthik Sridharan.
\newblock Learnability, stability and uniform convergence.
\newblock \emph{J. Mach. Learn. Res.}, 11:\penalty0 2635--2670, 2010.

\bibitem[Sodhani et~al.(2020)Sodhani, Chandar, and Bengio]{sodhani2020toward}
Shagun Sodhani, Sarath Chandar, and Yoshua Bengio.
\newblock Toward training recurrent neural networks for lifelong learning.
\newblock \emph{Neural computation}, 32\penalty0 (1):\penalty0 1--35, 2020.

\bibitem[Sodhani et~al.(2022)Sodhani, Faramarzi, Mehta, Malviya, Abdelsalam, Janarthanan, and Chandar]{sodhani2022introduction}
Shagun Sodhani, Mojtaba Faramarzi, Sanket~Vaibhav Mehta, Pranshu Malviya, Mohamed Abdelsalam, Janarthanan Janarthanan, and Sarath Chandar.
\newblock An introduction to lifelong supervised learning.
\newblock \emph{arXiv preprint arXiv:2207.04354}, 2022.

\bibitem[Song et~al.(2020{\natexlab{a}})Song, Jiang, and Ma]{SJM20}
Chaobing Song, Yong Jiang, and Yi~Ma.
\newblock Variance reduction via accelerated dual averaging for finite-sum optimization.
\newblock In Hugo Larochelle, Marc'Aurelio Ranzato, Raia Hadsell, Maria{-}Florina Balcan, and Hsuan{-}Tien Lin, editors, \emph{Advances in Neural Information Processing Systems 33: Annual Conference on Neural Information Processing Systems 2020, NeurIPS 2020, December 6-12, 2020, virtual}, 2020{\natexlab{a}}.

\bibitem[Song et~al.(2020{\natexlab{b}})Song, Jiang, and Ma]{song2020variance}
Chaobing Song, Yong Jiang, and Yi~Ma.
\newblock Variance reduction via accelerated dual averaging for finite-sum optimization.
\newblock \emph{Advances in Neural Information Processing Systems}, 33:\penalty0 833--844, 2020{\natexlab{b}}.

\bibitem[Tulving(1985)]{tulving1985many}
Endel Tulving.
\newblock How many memory systems are there?
\newblock \emph{American psychologist}, 40\penalty0 (4):\penalty0 385, 1985.

\bibitem[Wang et~al.(2019)Wang, Ji, Zhou, Liang, and Tarokh]{WJZLT19}
Zhe Wang, Kaiyi Ji, Yi~Zhou, Yingbin Liang, and Vahid Tarokh.
\newblock Spiderboost and momentum: Faster variance reduction algorithms.
\newblock In Hanna~M. Wallach, Hugo Larochelle, Alina Beygelzimer, Florence d'Alch{\'{e}}{-}Buc, Emily~B. Fox, and Roman Garnett, editors, \emph{Advances in Neural Information Processing Systems 32: Annual Conference on Neural Information Processing Systems 2019, NeurIPS 2019, December 8-14, 2019, Vancouver, BC, Canada}, pages 2403--2413, 2019.

\bibitem[Xiao and Zhang(2014)]{X14}
Lin Xiao and Tong Zhang.
\newblock A proximal stochastic gradient method with progressive variance reduction.
\newblock \emph{{SIAM} J. Optim.}, 24\penalty0 (4):\penalty0 2057--2075, 2014.

\bibitem[Yoon et~al.(2017)Yoon, Yang, Lee, and Hwang]{yoon2017lifelong}
Jaehong Yoon, Eunho Yang, Jeongtae Lee, and Sung~Ju Hwang.
\newblock Lifelong learning with dynamically expandable networks.
\newblock \emph{arXiv preprint arXiv:1708.01547}, 2017.

\bibitem[Zhang and Xiao(2019)]{zhang2019stochastic}
Junyu Zhang and Lin Xiao.
\newblock A stochastic composite gradient method with incremental variance reduction.
\newblock \emph{Advances in Neural Information Processing Systems}, 32, 2019.

\bibitem[Zhang et~al.(2013)Zhang, Mahdavi, and Jin]{zhang2013linear}
Lijun Zhang, Mehrdad Mahdavi, and Rong Jin.
\newblock Linear convergence with condition number independent access of full gradients.
\newblock \emph{Advances in Neural Information Processing Systems}, 26, 2013.

\bibitem[Zhou et~al.(2018)Zhou, Xu, and Gu]{ZXG18}
Dongruo Zhou, Pan Xu, and Quanquan Gu.
\newblock Stochastic nested variance reduced gradient descent for nonconvex optimization.
\newblock In Samy Bengio, Hanna~M. Wallach, Hugo Larochelle, Kristen Grauman, Nicol{\`{o}} Cesa{-}Bianchi, and Roman Garnett, editors, \emph{Advances in Neural Information Processing Systems 31: Annual Conference on Neural Information Processing Systems 2018, NeurIPS 2018, December 3-8, 2018, Montr{\'{e}}al, Canada}, pages 3925--3936, 2018.

\end{thebibliography}
\clearpage

\appendix
\section{Further Related Work}
\label{app:relatedwork}
In this section we present more extensively all work related to ours. This work lies in the intersection of convex finite sum minimization and incremental learning.

\textbf{Finite Sum Minimization:} As already mentioned, our work is closely related with the long line of research on variance reduction methods for strongly convex~\cite{NLST17,DBL14,KatyushaXzhu,allen2017katyusha,RSB12,X14,MGR17,johnson2013accelerating,allen2017katyusha,KatyushaXzhu,zhang2013linear,shalev2013accelerated,NLST17,DBL14,KatyushaXzhu,allen2017katyusha,RSB12,X14}, convex~\cite{lan2019unified,DVBS22,SJM20,LZ19,song2020variance,lin2018catalyst,delyon1993accelerated,lin2015accelerated,shalev2013stochastic,dubois2022svrg} and non-convex~\cite{allen2016improved,AZ18a,WJZLT19,PNPT20,LP21,LR21,FLLZ18,ZXG18,RHSPS16,LL18,KSADC22} finite-sum minimization (see also Appendix~\ref{app:relatedwork} for a more detailed discussion). As already mentioned the basic difference with our work comes from the fact these work concern the standard finite-sum setting while ours concerns its continual counter-part. 

Our work leverages two components of variance reduction techniques. Recent works in variance reduction such as the SVRG algorithm~\cite{johnson2013accelerating} and the Katyusha algorithm~\cite{allen2017katyusha,KatyushaXzhu} primarily leverage a full gradient  computation, which is essentially used as a substitute for the gradients that are not computed in this iteration, in order to reduce the variance of their gradient estimator. Older methods, which seek to accelerate SGD, such as the RPDG method \cite{lan2018optimal} and the SAG method \cite{roux2012stochastic}, which is a randomized variant of the IAG method \cite{blatt2007convergent}, use a stochastic average of the current gradient that it sampled and substitutes the rest of the gradients with ones computed from previous iterations. Our method seeks to utilize the first technique, but the complexity constraint does not allow for this tool. We mitigate this problem using a technique similar to the incremental gradient methods. There are also several other methods and techniques that relate to accelerated stochastic gradient methods, both in the convex setting \cite{song2020variance,allen2016improved,lan2019unified,lin2018catalyst,kesten1958accelerated,delyon1993accelerated,zhang2013linear,lin2015accelerated,shalev2013stochastic,shalev2013accelerated,dubois2022svrg} as well as in the non-convex setting \cite{zhang2019stochastic,fang2018spider}. In the past there have also been several works that focus on finite sum minimization in dynamic settings using SGD \cite{kowshik2021streaming,kowshik2021near,nagaraj2020least}

\textbf{Incremental Learning:} The other side of this work is incremental learning, the goal of which is adapting a model to new information, data or tasks without forgetting the ones that it was trained on earlier. The phenomenon of losing sight of old information is called \textit{catastrophic forgetting}~\cite{castro2018end,GMDCB13,kirkpatrick2017overcoming,mccloskey1989catastrophic} and it is one of the main challenges of incremental learning. The area closest to our work from the field of incremental learning is lifelong learning. In this area catastrophic forgetting is expressed in the "stability-plasticity dilemma" \cite{mermillod2013stability}. There have been three main approaches to tackle this issue, regularization based approaches \cite{nguyen2017variational}, memory based approaches \cite{tulving1985many} and architecture based approaches \cite{yoon2017lifelong,lopez2017gradient}. As shown in the current literature, combinations of these methods can be utilized in practice \cite{sodhani2020toward} as well. All of the previous directions are studied in depth in \cite{sodhani2022introduction} and a brief overview of the techniques of the field is also presented in \cite{castro2018end}. The area of lifelong learning also includes algorithms such as ELLA~\cite{ruvolo2013ella}

\section{Proof of Lemma~\ref{lemm:distbound2}}
\label{app:cons-optimum}
In this section we provide the proof of Lemma~\ref{lemm:distbound2}.
\distancebound*
The proof of Lemma~\ref{lemm:distbound2} follows by the proof of Lemma~\ref{lemm:optdist} that we state and prove up next. The proof of Lemma~\ref{lemm:distbound2} is deferred at the end of this section.

\begin{restatable}{lemma}{optdist}\label{lemm:optdist}
    For all $i \in [n-1]$ and $j \in [n-i]$,
    \[\norm{x^\star_{i+j}-x^\star_i}_2\leq \frac{2jG}{\mu(2i+j)}\]
where $x_i^\star = \argmin_{x \in \cD} g_i(x)$ and  $x_{i+j}^\star = \argmin_{x \in \cD} g_{i+j}(x)$. 
\end{restatable}

\begin{proof}[Proof of Lemma~\ref{lemm:optdist}]
    By the optimality of $x_{i+j}^\star\in \mathcal{D}$ for the function $g_{i+j}(x)$ we get that  
    
    \[\langle \nabla g_{i+j}(x_{i+j}^\star),x_i^\star-x_{i+j}^\star \rangle\geq 0\]

    The inequality is shown in \citep[equation 4.21]{boyd2004convex}.
    
    From the strong convexity of $g_{i+j}(x)$ we get that
    \begin{eqnarray*}
        \frac{\mu}{2}\norm{x_{i+j}^\star-x_i^\star}^2_2 &\leq& g_{i+j}(x_i^\star) - g_{i+j}(x_{i+j}^\star)- \langle \nabla g_{i+j}(x_{i+j}^\star),x_i^\star-x_{i+j}^\star\rangle\\
        &\leq& g_{i+j}(x_i^\star) -g_{i+j}(x_{i+j}^\star)\\
        &=&\frac{1}{i+j}\sum_{k=1}^{i+j} f_k(x_i^\star) -  \frac{1}{i+j}\sum_{k=1}^{i+j}f_k(x_{i+j}^\star)\\
        &=& \frac{i}{i+j}\left(g_i(x_i^\star)-g_i(x_{i+j}^\star)\right)+\frac{1}{i+j}\sum_{k=i+1}^{i+j}\left(f_k(x_i^\star)-f_k(x_{i+j}^\star)\right)\\
        &\leq& \frac{i}{i+j}\left(-\frac{\mu}{2}\norm{x_{i+j}^\star-x_i^\star}^2_2-\langle\nabla g_{i}(x_{i}^\star),x_{i+j}^\star - x_i^\star\rangle\right)\\ 
        &+&\frac{1}{i+j}\sum_{k=i+1}^{i+j}\left(f_k(x_i^\star)-f_k(x_{i+j}^\star)\right)\\
        &\leq& \frac{i}{i+j}\left(-\frac{\mu}{2}\norm{x_{i+j}^\star-x_i^\star}^2_2\right) +\frac{1}{i+j}\sum_{k=i+1}^{i+j}G\norm{x_{i+j}^\star-x_i^\star}_2\\
        &\leq& \frac{i}{i+j}\left(-\frac{\mu}{2}\norm{x_{i+j}^\star-x_i^\star}^2_2\right)+\frac{jG\norm{x_{i+j}^\star-x_i^\star}_2}{i+j}
    \end{eqnarray*}
\end{proof}
We conclude the section with the proof of Lemma~\ref{lemm:distbound2}.
\begin{proof}[Proof of Lemma~\ref{lemm:distbound2}]
By the inequality $\norm{a+b}_2^2 \leq 2\norm{a}^2_2 + 2\norm{b}^2_2$ we get that
\begin{eqnarray*}
        \norm{\hat{x}_j-x_i^\star}_2^2 &\leq& 2\left(\norm{x_i^\star-x_j^\star}_2^2+\norm{\hat{x}_j-x_j^\star}_2^2\right)\\
        &\leq& \frac{8}{\mu^2}\left(\frac{Gj}{2i+j}\right)^2+\norm{\hat{x}_j-x_j^\star}_2^2
\end{eqnarray*}
The second inequality comes from Lemma~\ref{lemm:optdist} and strong convexity.
\end{proof}

\section{Proof of Lemma~\ref{l:unbias} }
\label{app:bias}
In this section we provide the proof of Lemma~\ref{l:unbias}.
\unbias*
The basic step in order to establish Lemma~\ref{l:unbias} is Lemma~\ref{lemm:induction_var} that we present up next. The proof of Lemma~\ref{l:unbias} is deferred at the end of the section.
\inductvar*
\begin{proof}[Proof of Lemma~\ref{lemm:induction_var}]
    We will inductively establish Lemma~\ref{lemm:induction_var}. Notice that after stage $i=1$, Algorithm~\ref{alg:1} sets $\redvargrad_1 = \nabla f_1(\hat{x}_1)$. Up next we show that in case the induction hypothesis holds for stage $i-1$ then it must essentially hold for stage $i$. Up next we consider the following $3$ mutually exclusive cases:
    \begin{enumerate}
        \item $i-\prev\geq \alpha i$ meaning that $\prev$ is updated in this stage.
        \item $(i-1)-\prev\geq \alpha(i-1)$ meaning that $\prev$ was update in the previous stage.
        \item $i-\prev < \alpha i$ and $(i-1)-\prev < \alpha(i-1)$.
    \end{enumerate}
    For the first case Algorithm~\ref{alg:1} reaches Step~$5$ and thus $\redvargrad_{i-1} = \sum_{j=1}^{i-1} \nabla f_j(\hat{x}_{i-1})/i-1$. At the same time $\prev$ is set to $i-1$ meaning that $\prev=i-1$ and $\redvargrad_{i-1} = \frac{1}{i-1}\sum_{k=1}^{i-1} \nabla f_k(\hat{x}_{\prev})$.\\
    
    For the second case, Algorithm~\ref{alg:1} had reached Step~$12$ at stage $i-1$ meaning that $\redvargrad_{i-1} = \sum_{j=1}^{i-1} \nabla f_j(\hat{x}_{i-1})/i-1$. At the same time $\prev$ was set to $i-1$, meaning that $\prev=i-1$ and $\redvargrad_{i-1} = \frac{1}{i-1}\sum_{k=1}^{i-1} \nabla f_k(\hat{x}_{\prev})$.\\ 
    
    For the third case, from the inductive hypothesis we have that:
    \[\redvargrad_{i-2} = \frac{1}{i-2}\sum_{k=1}^{i-2} \nabla f_k(\hat{x}_{\prev})\]
    At stage $i-1$, $\redvargrad_{i-1}$ was calculated according to Step~$16$ of Algorithm~\ref{alg:1}. As a result, 
    \begin{align*}
        \begin{split}
            \redvargrad_{i-1} &= \left(1-\frac{1}{i-1}\right)\redvargrad_{i-2} +\frac{1}{i-1}\nabla f_{i-1}(\hat{x}_{\prev})\\
            &= \frac{i-2}{i-1}\frac{1}{i-2}\sum_{k=1}^{i-2} \nabla f_k(\hat{x}_{\prev})+ \frac{1}{i-1}\nabla f_{i-1}(\hat{x}_{\prev})\\
            &= \frac{1}{i-1}\sum_{k=1}^{i-1} \nabla f_k(\hat{x}_{\prev})
        \end{split}
    \end{align*}
\end{proof}

\begin{proof}[Proof of Lemma~\ref{l:unbias}]
Let $\mathcal{F}_t^i$ denote the filtration until step $t \in [T_i]$ of stage $i \in [n]$.

By the definition of $\nabla_i^t$ at Step~$5$ of Algorithm~\ref{alg:frequent} we know that

\begin{eqnarray*}
        \bE\left[\nabla_i^t~|~\mathcal{F}_t^i\right] &=& \left(1-\frac{1}{i}\right)\left(\bE\left[\nabla f_{u_t}(x_i^t)-\nabla f_{u_t}(\hat{x}_{\prev})~|~\mathcal{F}_t^i\right]+\redvargrad_{i-1}\right)+\frac{1}{i}\nabla f_i(x_i^t)\\
        &=& \frac{i-1}{i}\left(\sum_{k=1}^{i-1} \mathrm{Pr}\left[u_t = k\right]\left(\nabla f_{k}(x_i^t)-\nabla f_k(\hat{x}_{\prev})\right)+\redvargrad_{i-1}\right)+\frac{1}{i}\nabla f_i(x_i^t)\\
        &=& \frac{i-1}{i}\left(\frac{1}{i-1}\sum_{k=1}^{i-1}\left(\nabla f_{k}(x_i^t)-\nabla f_k(\hat{x}_{\prev})\right)
         +\redvargrad_{i-1}\right)+\frac{1}{i}\nabla f_i(x_i^t)\\
        &=& \frac{1}{i}\sum_{k=1}^{i-1}\nabla f_{k}(x_i^t) - \frac{1}{i}\sum_{k=1}^{i-1}\nabla f_k(\hat{x}_{\prev}) + \frac{i-1}{i}\redvargrad_{i-1} +\frac{1}{i}\nabla f_i(x_i^t)
\end{eqnarray*}
Lemma~\ref{lemm:induction_var} ensures that $\frac{i-1}{i}\redvargrad_{i-1}:= \frac{1}{i}\sum_{k=1}^{i-1} \nabla f_k(\hat{x}_{\prev})$ meaning that
\[\bE\left[\nabla_i^t~|~\mathcal{F}_t^i\right] = \frac{1}{i}\sum_{k=1}^{i-1}\nabla f_{k}(x_i^t) +\frac{1}{i}\nabla f_i(x_i^t) = \nabla g_i(x_i^t).\]
\end{proof}

\section{Proof of Lemma~\ref{l:bounded_variance}}
\label{app:variance}
In this section we provide the proof of Lemma~\ref{l:bounded_variance}.
\boundvariance*
In order to prove Lemma~\ref{l:bounded_variance} we first state and establish Lemma~\ref{lemm:variance}
\variance*
\begin{proof}[Proof of Lemma~\ref{lemm:variance}]
By substituting the definition of $\nabla_i^t$:
\begin{align*}
    \begin{split}
        &\bE[\norm{\nabla_i^t-\nabla g_i(x_i^t)}_2^2] = \left(1-\frac{1}{i}\right)^2 \cdot  \bE\left[\norm{\nabla f_{u_t}(x_i^t)-\nabla f_{u_t}(\hat{x}_{\prev})+\redvargrad_{i-1}-\nabla g_{i-1}(x_i^t)}_2^2\right]\\
        &\leq 2\left(1-\frac{1}{i}\right)^2 \cdot \bE\left[\norm{\nabla f_{u_t}(x_i^t)-\nabla f_{u_t}(\hat{x}_{\prev})}_2^2\right] + 2\left(1-\frac{1}{i}\right)^2 \cdot \bE\left[\norm{\frac{1}{i-1}\sum_{k=1}^{i-1} \nabla f_k(x_i^t) -\redvargrad_{i-1}}_2^2\right]\\
        &= 2\left(1-\frac{1}{i}\right)^2 \cdot \bE\left[\norm{\nabla f_{u_t}(x_i^t)-\nabla f_{u_t}(\hat{x}_{\prev})}_2^2\right]\\
        &+ 2\left(1-\frac{1}{i}\right)^2 \cdot \bE\left[\norm{\frac{1}{i-1}\sum_{k=1}^{i-1} (\nabla f_k(x_i^t) - \nabla f_k(\hat{x}_{\prev}))}_2^2\right]\\
        &\leq 2\left(1-\frac{1}{i}\right)^2 \cdot L^2 \cdot \bE\left[\norm{x_i^t-\hat{x}_{\prev}}_2^2\right]+ 2\left(1-\frac{1}{i}\right)^2 \frac{1}{(i-1)^2}\bE\left[\norm{\sum_{k=1}^{i-1} (\nabla f_k(x_i^t) - \nabla f_k(\hat{x}_{\prev}))}_2^2\right]\\
        &\leq 2\left(1-\frac{1}{i}\right)^2 \cdot L^2 \cdot \bE\left[\norm{x_i^t-\hat{x}_{\prev}}_2^2\right]+ 2\left(1-\frac{1}{i}\right)^2 \frac{i-1}{(i-1)^2}\bE\left[\sum_{k=1}^{i-1}\norm{ \nabla f_k(x_i^t) - \nabla f_k(\hat{x}_{\prev})}_2^2\right]\\
                &= 2\left(1-\frac{1}{i}\right)^2 \cdot L^2 \cdot \bE\left[\norm{x_i^t-\hat{x}_{\prev}}_2^2\right]+ 2\left(1-\frac{1}{i}\right)^2 \frac{1}{(i-1)}\sum_{k=1}^{i-1} \bE\left[ \norm{ \nabla f_k(x_i^t) - \nabla f_k(\hat{x}_{\prev})}^2_2\right]\\
                                &= 2\left(1-\frac{1}{i}\right)^2 \cdot L^2 \cdot \bE\left[\norm{x_i^t-\hat{x}_{\prev}}_2^2\right]+ 2\left(1-\frac{1}{i}\right)^2 \frac{1}{(i-1)}\sum_{k=1}^{i-1} \bE\left[ \norm{ \nabla f_k(x_i^t) - \nabla f_k(\hat{x}_{\prev})}^2_2\right]\\
        &\leq 2\left(1-\frac{1}{i}\right)^2  \cdot L^2 \cdot \bE\left[\norm{x_i^t-\hat{x}_{\prev}}_2^2\right]+2\left(1-\frac{1}{i}\right)^2 \frac{L^2}{(i-1)}\sum_{k=1}^{i-1}\bE\left[\norm{x_i^t - \hat{x}_{\prev}}_2^2\right]\\
        &= 4\left(1-\frac{1}{i}\right)^2 L^2\bE\left[\norm{x_i^t - \hat{x}_{\prev}}_2^2\right]\\
        &\leq 8L^2\left(1-\frac{1}{i}\right)^2\left(\bE[\norm{x_i^t - x_i^\star}_2^2]+\bE[\norm{x_i^\star - \hat{x}_{\prev}}_2^2]\right)\\
    \end{split}
\end{align*}
\end{proof}
\begin{proof}[Proof of Lemma~\ref{l:bounded_variance}]
Applying Lemma~\ref{lemm:distbound2} for $j:= \prev$ we get that,
\begin{eqnarray*}
        \bE\left[\norm{x_i^\star - \hat{x}_{\prev}}_2^2\right] &\leq& \frac{8}{\mu^2}\left(\frac{G(i-\prev)}{i+\prev}\right)^2+2\bE\left[\norm{x_{\prev}^\star-\hat{x}_{\prev}}^2_2\right]\\
        &\leq& \frac{8}{\mu^2}\left(\frac{G(i-\prev)}{i}\right)^2+2\bE\left[\norm{x_{\prev}^\star-\hat{x}_{\prev}}^2_2\right]\\
        &\leq& \frac{8}{\mu^2}\left(\frac{Gi\alpha}{i}\right)^2+2\bE\left[\norm{x_{\prev}^\star-\hat{x}_{\prev}}^2_2\right]\\
        &=& \frac{8G^2}{\mu^2}\alpha^2+2\bE\left[\norm{x_{\prev}^\star-\hat{x}_{\prev}}^2_2\right]
\end{eqnarray*}
The third inequality comes from the fact that $i-\prev\leq \alpha i$ which is enforced by Step~$4$ of Algorithm~\ref{alg:1}. 

By Lemma~\ref{lemm:variance} we get that
\begin{eqnarray*}
\bE[\norm{\nabla_i^t-\nabla g_i(x_i^t)}_2^2] &\leq& 8L^2\bE[\norm{x_i^t - x_i^\star}_2^2]+8L^2\bE[\norm{x_i^\star - \hat{x}_{\prev}}_2^2]\\
&\leq& 8L^2\bE[\norm{x_i^t - x_i^\star}_2^2] + \frac{64G^2L^2}{\mu^2}\alpha^2+16 L^2\bE\left[\norm{x_{\prev}^\star-\hat{x}_{\prev}}^2_2\right]
\end{eqnarray*}
which concludes the proof.
\end{proof}

\section{Proof of Lemma~\ref{lemm:convergebound}}
\label{app:convbound}
In this section we provide proof for Lemma~\ref{lemm:convergebound}.

\begin{restatable}[Convergence Bound]{lemma}{convergebound}
    \label{lemm:convergebound} If $\bE[g_j(\hat{x}_j)]-g_j(x_j^\star)\leq \epsilon$ for all stages $j \in [i-1]$ then 
    \begin{align}\label{eq:convergencebound}
        \begin{split}
            \bE\left[g_i\left(\hat{x}_i\right)-g_i(x_i^\star)\right]&\leq \frac{1}{\cZ}\left(\frac{\mu}{8}(\beta-1)(\beta-2)\bE\left[\norm{x_i^0-x_i^\star}_2^2\right]+\kappa^2G^2\alpha^2\frac{1}{\mu}T_i+2\kappa^2 \epsilon T_i\right)
        \end{split}
    \end{align}
    where $\cZ = \frac{T_i(T_i-1)}{2}+(T_i+1)(\beta-1)$, $\kappa = 6L/\mu$ and $\beta = 72L^2/\mu^2$.
\end{restatable}
\begin{proof}[Proof of Lemma~\ref{lemm:convergebound}]
In order to establish Lemma~\ref{lemm:convergebound} we use Lemma~\ref{lemm:classicbound} which is similar to \citep[Lemma 3.2]{KatyushaXzhu}. The proof of Lemma~\ref{lemm:classicbound} is presented in Appendix~\ref{app:classic-bound}.
\begin{restatable}{lemma}{classicbound}
        \label{lemm:classicbound}
        In case $\bE[\nabla_i^t] = \nabla g_i(x_i^t)$ then
        \begin{align}\label{eq:main-start}
            \begin{split}
                \bE\left[g_i(x_i^{t+1})-g_i(x_i^\star)\right]\leq \bE\left[\frac{9}{16}\gamma_t\norm{\nabla_i^t-\nabla g_i(x_i^t)}^2_2
                +\frac{1-\mu \gamma_t}{2\gamma_t}\norm{x_i^\star-x_i^t}^2_2
                -\frac{1}{2\gamma_t}\norm{x_i^\star-x_i^{t+1}}^2_2\right]
            \end{split}
        \end{align}
        where $x_i^{t+1} := \Pi_\cD \left(x_i^t - \gamma_t\nabla_i^t\right)$, $\gamma_t = \frac{4}{\mu(t+\beta)}$ and $\beta= 72L^2/\mu^2$ (see Step~$1$ of Algorithm~\ref{alg:frequent}).
    \end{restatable}
    \unbias*
Since Lemma~\ref{l:unbias} establishes that $\bE[\nabla_i^t] = \nabla g_i(x_i^t)$ we can apply Lemma~\ref{lemm:classicbound}. In order to bound the variance $\bE\left[\norm{\nabla_i^t-\nabla g_i(x_i^t)}^2_2 \right]$ appearing in the RHS of Equation~\ref{eq:main-start}, we use Lemma~\ref{l:bounded_variance} (see Appendix~\ref{app:variance} for its proof).
\boundvariance*
As discussed above by combining Lemma~\ref{lemm:classicbound} with Lemma~\ref{l:bounded_variance}, we get the following:
\begin{eqnarray}
    \bE\left[g_i(x_i^{t+1})-g_i(x_i^\star)\right] &\leq& \frac{\gamma_t^{-1}+9L^2\gamma_t-\mu}{2} \bE\left[\norm{x_i^t - x_i^\star}_2^2\right] - \frac{\gamma_t^{-1}}{2}\bE\left[\norm{x_i^{t+1} - x_i^\star}_2^2\right]\nonumber\\
    &+& 36\frac{L^2G^2}{\mu^2}\alpha^2\gamma_t + 9L^2\gamma_t\bE\left[\norm{x^\star_{\prev}-\hat{x}_{\prev}}^2_2\right]
    \label{eq:main-2}
\end{eqnarray}
By the strong-convexity of $g_{\prev}(\cdot)$ we get 
\[\norm{x^\star_{\prev}-\hat{x}_{\prev}}^2_2 \leq \frac{2}{\mu}\left(g_{\prev}(\hat{x}_{\prev})-g_{\prev}(x^{\star}_{\prev})\right)\]
and from our inductive hypothesis
\[\bE\left[g_{\prev}(\hat{x}_{\prev})-g_{\prev}(x^{\star}_{\prev})\right] \leq \epsilon\]
Notice that by the selection of $\gamma_t=4/(\mu(t+\beta))$ and $\beta=72L^2/\mu^2$ of Algorithm~\ref{alg:frequent} we get that
\begin{align*}
    \begin{split}
        9L^2\gamma_t= \frac{36L^2}{\mu(t+\beta)} \leq \frac{36L^2}{\mu \beta}
        \leq \frac{36 L^2}{\mu72L^2/\mu^2} \leq \frac{36 L^2}{72L^2/\mu} = \frac{\mu}{2}
    \end{split}
\end{align*} 
Using the previous inequalities in Equation~\ref{eq:main-2} we get that
    \begin{eqnarray}\label{eq:3}
            \bE\left[g_i(x_{t+1})-g_i(x_i^\star)\right] &\leq& \frac{\gamma_t^{-1}-\mu/2}{2} \bE\left[\norm{x_i^t - x_i^\star}_2^2\right] - \frac{\gamma_t^{-1}}{2}\bE\left[\norm{x_i^{t+1} - x_i^\star}_2^2\right]\nonumber\\
        &+& \kappa^2G^2\alpha^2\gamma_t + \frac{18L^2\gamma_t}{\mu}\epsilon  
    \end{eqnarray}
Since $\kappa = 6L/\mu$.

Multiplying both parts of Equation~\ref{eq:3} with $(t+\beta-1)$ and substituting $\gamma_t=4/\left(\mu(t+\beta)\right)$ we get that

\begin{eqnarray*}
     (t+\beta-1)\bE\left[g_i(x_i^{t+1})-g_i(x_i^\star)\right] &\leq& \frac{\mu(t+\beta-1)(t+\beta-2)}{8}\bE\left[\norm{x_i^t - x_i^\star}_2^2\right]\nonumber\\ 
            &-& \frac{\mu(t+\beta-1)(t+\beta)}{8}\bE\left[\norm{x_i^{t+1} - x_i^\star}_2^2\right] \nonumber\\
            &+& \left(\kappa^2G^2\alpha^2 +\frac{18L^2}{\mu}\epsilon\right)\cdot  \frac{4}{\mu(t+\beta)} \cdot (t + \beta - 1)
\end{eqnarray*}
where $\beta = 72L^2/\mu^2$. By setting $\kappa = 6L/\mu$ we get that,
\begin{eqnarray*}
            (t+\beta-1)\bE\left[g_i(x_i^{t+1})-g_i(x_i^\star)\right] &\leq& \frac{\mu(t+\beta-1)(t+\beta-2)}{8}\bE\left[\norm{x_i^t - x_i^\star}_2^2\right]\nonumber\\ 
            &-& \frac{\mu(t+\beta-1)(t+\beta)}{8}\bE\left[\norm{x_i^{t+1} - x_i^\star}_2^2\right] \nonumber\\
            &+& \left(\kappa^2\frac{G^2}{\mu}\alpha^2 +2\kappa^2\epsilon\right)    
\end{eqnarray*}
By taking the summation over all iterations $t=0,\dots,T_i-1$, we get
\begin{eqnarray*}
            \sum_{t=0}^{T_i-1}(t+\beta-1)\bE[g_i(x_i^{t+1})-g_i(x_i^\star)] &\leq& \frac{\mu}{8}\sum_{t=0}^{T_i-1}(t+\beta-1)(t+\beta-2)\bE\left[\norm{x_i^t-x_i^\star}_2^2\right]\\
            &-&\frac{\mu}{8}\sum_{t=0}^{T_i-1}(t+\beta-1)(t+\beta)\bE\left[\norm{x_i^{t+1}-x_i^\star}_2^2\right]\\
            &+&T_i\left(\kappa^2\frac{G^2}{\mu}\alpha^2 +2\kappa^2\epsilon\right)
\end{eqnarray*}
As a result, we get that 
\begin{eqnarray*}
            \sum_{t=0}^{T_i-1}(t+\beta-1)\bE[g_i(x_i^{t+1})-g_i(x_i^\star)] &\leq& \frac{\mu}{8}(\beta-1)(\beta-2)\bE\left[\norm{x_i^0-x_i^\star}_2^2\right] \\
            &-&\frac{\mu}{8}(T_i+\beta-2)(T_i+\beta-1)\bE\left[\norm{x_{T_i}-x_i^\star}_2^2\right]\\
            &+&T_i\left(\kappa^2\frac{G^2}{\mu}\alpha^2 +2\kappa^2\epsilon\right)   
\end{eqnarray*}
Dividing by $\cZ = \sum_{t=0}^{T_i-1}(t+\beta-1) = T_i(T_i-1)/2+T_i(\beta-1)$ and using the convexity of $g_i$ we get that
\begin{eqnarray*}
    \bE\left[g_i\left(\frac{1}{\cZ}\sum_{t=0}^{T_i-1}(t+\beta-1)x_i^{t+1}\right)-g_i(x_i^\star)\right]&\leq& \frac{1}{\cZ}\frac{\mu}{8}(\beta-1)(\beta-2)\bE\left[\norm{x_i^0-x_i^\star}_2^2\right]+\kappa^2\frac{G^2}{\mu}\alpha^2T_i\frac{1}{\cZ}\\
            &+&2\kappa^2\frac{T_i \epsilon}{\cZ}    
\end{eqnarray*}
As a result,
\[\bE\left[g_i\left(\hat{x}_i\right)-g_i(x_i^\star)\right]\leq \frac{1}{\cZ}\left(\frac{\mu}{8}(\beta-1)(\beta-2)\bE\left[\norm{x_i^0-x_i^\star}_2^2\right]+\kappa^2G^2\alpha^2\frac{1}{\mu}T_i+2\kappa^2 \epsilon T_i\right)\]
\end{proof}

\subsection{Proof of  Lemma~\ref{lemm:classicbound}}\label{app:classic-bound}
In this section we provide the proof for Lemma~\ref{lemm:classicbound}.
\classicbound*
\begin{proof}[Proof of Lemma~\ref{lemm:classicbound}]
    We start with \citep[Lemma 3.2]{KatyushaXzhu}, which we state below.
    \begin{restatable}{lemma}{zhubound}\cite{KatyushaXzhu}
        If $w_{t+1}=\argmin_{y\in \bR^d}\{\frac{1}{2\eta}\norm{y-w_t}^2_2+\psi(y)+\langle\redvargrad_t,y\rangle\}$ for some random vector $\redvargrad_t \in \bR^d$ satisfying $\bE[\redvargrad_t]=\nabla f(w_t)$, then for every $u\in \bR^d$, we have
        \[\bE[F(w_{t+1})-F(u)]\leq \bE\left[\frac{\eta}{2(1-\eta L)}\norm{\redvargrad_t-\nabla f(w_t)}^2_2 + \frac{(1-\mu_f \eta)\norm{u-w_t}^2_2-(1+\mu_\psi\eta)\norm{u-w_{t+1}}^2_2}{2\eta}\right]\]
    where $\psi$ is a proximal term, $\mu_\psi$ is its strong convexity and $\mu_f$ is the strong convexity of the optimization function $f$ and $\eta$ is the step of the algorithm and the function $F$ is defined as:
    \[F(x) = f(x) + \psi(x)\]
    \end{restatable}
    For our setting the proximal term is:
    \[\psi(x) = \begin{cases}
            0, & x \in \cD\\
            \infty, & x \notin \cD
                \end{cases}\]
    This means that $\mu_\psi = 0$.
    The step $\eta$ for our analysis is $\gamma_t$ and $f\equiv g_i$. 
    
    By taking $u = x_i^\star = \argmin_{x\in \cD} g_i(x)$ and since due to projection on $\cD$
    \[F(x_i^{t+1})=g_i(x_i^{t+1})+\psi(x_i^{t+1})=g_i(x_i^{t+1})\] 
    the inequality can be restated as:
    \[\bE[g_i(x_i^{t+1})-g_i(x_i^\star)]\leq \bE\left[\frac{\gamma_t}{2(1-\gamma_t L)}\norm{\nabla_i^t - \nabla g_i(x_i^t)}^2_2+\frac{1-\mu \gamma_t}{2\gamma_t}\norm{x_i^\star-x_i^t}^2_2
                -\frac{1}{2\gamma_t}\norm{x_i^\star-x_i^{t+1}}^2_2\right]\]
    Notice that by the selection of $\gamma_t = 4/\left(\mu(t+\beta)\right)$ in Step~$6$ of Algorithm~\ref{alg:frequent} we can do the following simplification
    \[\frac{1}{(1-\gamma_tL)}=\frac{1}{(1-\frac{4L}{\mu(t+\beta)})}\leq \frac{1}{(1-\frac{4L}{\mu\beta})}=\frac{1}{(1-\frac{4\mu}{72L})}\leq \frac{1}{(1-\frac{4}{72})}=\frac{18}{17}\leq \frac{9}{8}\]
    which gives the theorem statement. The last part of the proof required, is to show that the following two update rules are equivalent.
    \begin{align*}
        \begin{split}
            x_i^{t+1}&=\argmin_{y\in \bR^d}\{\frac{1}{2\gamma_t}\norm{y-x_i^t}^2_2+\psi(y)+\langle\nabla_i^t,y\rangle\}\\
            x_i^{t+1} &= \Pi_\cD \left(x_i^t - \gamma_t\nabla_i^t\right) 
        \end{split}
    \end{align*}
    which is a well known fact in the literature and this concludes the proof.
\end{proof}

\section{Omitted Proofs of Section~\ref{section:convergence-results}}
\subsection{Proof of Lemma~\ref{l:sparse}}\label{app:complex}
In this section we prove Lemma~\ref{l:sparse}, for the sake of exposition we restate it up next.
\sparse*
\begin{proof}[Proof of Lemma~\ref{l:sparse}]
Step~$5$ and~$12$ are only executed when the following inequality is satisfied:
    \begin{align}\label{eq:cons-seq}
        \begin{split}
            i-\prev \geq \alpha\cdot i \Rightarrow
            i \geq \frac{1}{1-\alpha} \cdot \prev
        \end{split}
    \end{align}
    Once Algorithm~\ref{alg:1} reaches Step~$5$ and~$12$ it necessarily, reaches Step~$13$ where $\prev$ is updated to $i$. Let $z_0=1,z_1,\ldots,z_k,\ldots$ the sequence of stages where $z_k$ denotes the stage at which Algorithm~\ref{alg:1} reached Step~$5$ and~$12$ for the $k$-th time. By Equation~\ref{eq:cons-seq} we get that $z_{k+1} \geq \frac{1}{1-\alpha} \cdot z_k$ implying that 
    \[     z_k \geq \left(\frac{1}{1-\alpha}\right)^k \]
    Since $z_k \leq n$ we get that $k \leq \frac{\log n}{ \log \left(\frac{1}{1-\alpha}\right)}$. Notice that $\log \left(\frac{1}{1-\alpha}\right) = -\log(1-\alpha) \geq 1-(1-\alpha) = \alpha$ and thus $ k \leq \frac{\log n}{ \alpha}$.   
\end{proof}

Using Lemma~\ref{l:sparse}, we can now also show Corollary~\ref{l:FOs}.
\fos*
\begin{proof}[Proof of Corollary~\ref{l:FOs}]
    At each iteration of Algorithm~\ref{alg:frequent} requires $3$ FOs (Step~$5$) and thus Algorithm~\ref{alg:frequent} requires overall $3T_i$ FOs during stage $i \in [n]$. At Step~$5$ and~$12$ Algorithm~\ref{alg:1} requires at most $n$ FOs and thus by Lemma~\ref{l:sparse} it overall requires $2n\lceil\log n/\alpha \rceil$ FOs.
\end{proof}

\subsection{Proof of Theorem~\ref{t:accuracy}}\label{app:thm3}
In this section we provide the proof for Theorem~\ref{t:accuracy}, for the sake of exposition we restate it up next.
\accuracy*

\begin{proof}[Proof of Theorem~\ref{t:accuracy}]
    At stage $i:=1$, Algorithm~\ref{alg:1} performs gradient descent using $f_1$ in order to produce $\hat{x}_1 \in \mathcal{D}$. The latter requires $\cO\left(L\log(1/\epsilon)/\mu\right)$ FOs.
    Let us inductively assume that,
    \[\bE[g_j(\hat{x}_j)]-g_j(x_j^\star)\leq \epsilon ~~~\text{ for all }~~j \in [i-1].\]
    Using the latter we will establish that $\bE[g_i(\hat{x}_i)]-g_i(x_i^\star)\leq \epsilon$. In order to do the latter we first use Lemma~\ref{lemm:convergebound}. The proof of which can be found in Appendix~\ref{app:convbound} and its proof is based on the fact that the estimator $\nabla_i^t$ is unbiased and admits bounded variance (see Lemma~\ref{l:unbias} and~\ref{l:bounded_variance}).
    \convergebound*
    Since the conditions of Lemma~\ref{lemm:convergebound} are ensured by the induction hypothesis, we will appropriately select $T_i$ such that the right-hand side of 
    Equation~\ref{eq:convergencebound} is upper bound by $\epsilon > 0$.
    
    At first, we upper bound the term $\bE\left[\norm{x_i^0-x_i^
    \star}^2_2\right]$ appearing in the first term of the RHS of Equation~\ref{eq:convergencebound}. Recall that in Step~$2$ of Algorithm~\ref{alg:frequent}, we set $x_i^0\leftarrow \hat{x}_{i-1}\in \cD$. As a result, 
    \[\norm{x_i^0-x_i^
    \star}^2_2 = \norm{\hat{x}_{i-1}-x_i^
    \star}^2_2 \]
    In order to upper bound the term $\norm{\hat{x}_{i-1}-x_i^
    \star}^2_2$ we use Lemma~\ref{lemm:distbound2}, the proof of which can be found in Appendix~\ref{app:cons-optimum}.
    \distancebound*
    Applying Lemma~\ref{lemm:distbound2} for $j := i-1$ we get that 
    \begin{equation}
        \bE\left[\norm{x_i^0-x_i^\star}_2^2\right] = \bE\left[\norm{\hat{x}_{i-1}-x_i^\star}_2^2\right] \leq \frac{8}{\mu^2}\left(\frac{G}{2i-1}\right)^2+2\bE\left[\norm{\hat{x}_{i-1}-x^\star_{i-1}}_2^2\right]
    \end{equation}
    By the inductive hypothesis we know that $\bE\left[ g_{i-1}(\hat{x}_{i-1})\right]-g_{i-1}(x^\star_{i-1}) \leq \epsilon$ and thus by the strong convexity $g_i(\cdot)$, we get that
    \[\bE\left[\norm{\hat{x}_{i-1}-x^\star_{i-1}}^2_2\right] \leq \frac{2}{\mu}\epsilon\]
    which yields Equation~\ref{eq:new2}:
    \begin{equation}\label{eq:new2}
        \bE\left[\norm{x_i^0-x_i^\star}_2^2\right] \leq \frac{8}{\mu^2}\left(\frac{G}{2i-1}\right)^2+\frac{4}{\mu}\epsilon
    \end{equation}
The value of $\cZ$ can be lower bounded as follows:
    \[ \cZ := T_i(T_i-1)/2+(T_i+1)(\beta-1) \geq T_i(T_i-1)/2 \geq T_i^2/4\]
    So by combining Equation~\ref{eq:convergencebound} with the previous two inequalities we get:
    \begin{eqnarray*}
        \bE[g_i(\hat{x}_i)]-g_i(x_i^\star) &\leq& \frac{4}{T_i^2}\left(\frac{\mu}{8}(\beta-1)(\beta-2)\left(\frac{8}{\mu^2}\left(\frac{G}{2i-1}\right)^2+\frac{4}{\mu}\epsilon\right)+\kappa^2G^2\alpha^2\frac{1}{\mu}T_i+2\kappa^2\epsilon T_i\right)\\
        &=& \left(\frac{4G^2}{\mu T_i^2(2i-1)^2}+\frac{2\epsilon}{T_i^2}\right)(\beta-1)(\beta-2)+4\kappa^2G^2\frac{\alpha^2}{T_i\mu}+8\kappa^2\epsilon\frac{1}{T_i}
    \end{eqnarray*}
    
    Now the upper bound on $\bE[g_i(\hat{x}_i)]-g_i(x_i^\star)$ admits four terms all of which depend on $T_i$. Thus we can select $T_i$ so that each of the terms is upper bounded by $\epsilon / 4$. Namely,
    
    \[
    \begin{cases}
        \frac{4G^2}{\mu T_i^2(2i-1)^2}(\beta-1)(\beta-2) &\leq \epsilon/4\\
        \frac{2\epsilon}{T_i^2}(\beta-1)(\beta-2) &\leq \epsilon/4\\
        4\kappa^2G^2\frac{\alpha^2}{i^2T_i\mu} &\leq \epsilon/4\\
        8\kappa^2\epsilon\frac{1}{T_i}&\leq \epsilon/4
    \end{cases}
    \]
    Since $\kappa=6\frac{L}{\mu}$, we get that we can set $T_i$ as follows,
    \begin{eqnarray*}
        T_i &=& \max\{192\frac{GL^2}{\mu^{5/2}i\sqrt{\epsilon}}, 204\frac{L^2}{\mu^2}, 576\frac{L^2G^2\alpha^2}{\mu^3 i^2 T_i}, 1152\frac{L^2}{\mu^2}\}\\
        &=& \max\{192\frac{GL^2}{\mu^{5/2}i\sqrt{\epsilon}}, 576\frac{L^2G^2\alpha^2}{\mu^3 i^2 T_i}, 1152\frac{L^2}{\mu^2}\}\\
        &\leq& 192\frac{GL^2}{\mu^{5/2}i\sqrt{\epsilon}}+ 576\frac{L^2G^2\alpha^2}{\mu^3 i^2 T_i} + 1152\frac{L^2}{\mu^2}
    \end{eqnarray*}
    The proof is completed by selecting  $\alpha=\mu\epsilon^{1/3}/(9G^{2/3}L^{2/3})$.
\end{proof}

\section{Proof of Theorem~\ref{thm:1}}\label{app:const-approx}
In this section we provide the proof for Theorem~\ref{thm:1}.
\static*
\begin{proof}[Proof of Theorem~\ref{thm:1}]
The proof of Theorem~\ref{thm:1} follows by combining Theorem~\ref{t:accuracy} and Corollary~\ref{l:FOs}. Their proofs are respectively in Appendix~\ref{app:thm3} and~\ref{app:complex}.

\accuracy*
\fos*
Using the selection of $T_i$ provided in Theorem~\ref{t:accuracy} to Corollary~\ref{l:FOs} we get that
\begin{eqnarray*}
    \sum_{i=1}^n T_i &\leq& \frac{720GL^2}{\mu^{5/2}\sqrt{\epsilon}}\sum_{i=1}^n\frac{1}{i}+ n \left(\frac{9L^{2/3}G^{2/3}}{\epsilon^{1/3}\mu}+864\frac{L^2}{\mu^2}\right)\\
    &\leq& \frac{720GL^2}{\mu^{5/2}\sqrt{\epsilon}}\log n+n\left(\frac{9L^{2/3}G^{2/3}}{\epsilon^{1/3}\mu}+864\frac{L^2}{\mu^2}\right)
\end{eqnarray*}
At the same time, using the selection of $\alpha$ provided in Theorem~\ref{t:accuracy} we get
\begin{eqnarray*}
    2n\log n /\alpha:= 40\frac{L^{2/3}G^{2/3}}{\mu}\frac{n\log n}{\epsilon^{1/3}}
\end{eqnarray*}
\end{proof}

\section{Proof of Theorem~\ref{c:imp}}\label{app:lowerbound}
In this section we provide the proof of Theorem~\ref{c:imp}. We first present an intermediate theorem that is the main technical contribution of the section. 

\begin{restatable}{theorem}{lowerbound}\label{thm:lowerbound}
Let a natural first-order method $\mathcal{A}$ (even randomized) that given a sequence of $n$ strongly convex functions $f_1,\ldots,f_n$ outputs a sequence $\hat{x}_1,\ldots,\hat{x}_n \in \mathcal{D}$ by performing overall $o(n^2)$ FOs. Then there exists a sequence of strongly convex functions $f_1,\ldots,f_n$ with $f_i:[-1,1]^d \mapsto \mathbb{R}$ and $\mu,G,L = \mathcal{O}(1)$ such that
\[ \bE\left[g_i(\hat{x}_i)\right] - g_i(x_i^\star) \geq \Omega\left(1/n^4\right) ~~~~\text{ for some stage } i \in [n]\]
\end{restatable}

The proof of Theorem~\ref{thm:lowerbound} lies in Section~\ref{s:thm:lowerbound}. To this end we use Theorem~\ref{thm:lowerbound} to establish Theorem~\ref{c:imp}.

\imp*
\begin{proof}[Proof of Theorem~\ref{c:imp}]
Let us assume that there exists a natural first order method $\mathcal{A}$ with overall complexity $\mathcal{O}\left(n^{2-\alpha}\log(1/\epsilon)\right)$ for some $\alpha >0$. By setting $\epsilon = \mathcal{O}(1/n^5)$, we get that there exists a natural first-order method that for sequence $f_1,\ldots,f_n$ guarantees
\[ \bE\left[g_i(\hat{x}_i)\right] - g_i(x_i^\star) \leq \mathcal{O}\left(1/n^5\right) ~~~~\text{ for each stage } i \in [n]\]
with overall FO complexity $\mathcal{O}(n^{2-\alpha}\log n)$. However the latter contradicts with Theorem~\ref{thm:lowerbound}.

Respectively let us assume that there exists a natural first order method $\mathcal{A}$ with overall complexity $\mathcal{O}\left(n/\epsilon^\alpha \right)$ for some $\alpha < 1/4$. By setting $\epsilon = \mathcal{O}(1/n^4)$, we get that there exists a natural first-order method that for sequence $f_1,\ldots,f_n$ guarantees
\[ \bE\left[g_i(\hat{x}_i)\right] - g_i(x_i^\star) \leq \mathcal{O}\left(1/n^4\right) ~~~~\text{ for each stage } i \in [n]\]
with overall FO complexity $\mathcal{O}(n^{1+\alpha/4})$. However the latter contradicts with Theorem~\ref{thm:lowerbound}. 
\end{proof}

\subsection{Proof of Theorem~\ref{thm:lowerbound}}\label{s:thm:lowerbound}

\begin{proof}[Proof of Theorem~\ref{thm:lowerbound}]
To simplify notation, we denote with $[x]_\ell$ the coordinate $\ell$ of vector $x \in \mathbb{R}^d$. In our lower bound construction we consider $d:=2$.

Since $\mathcal{A}$ performs $o(n^2)$ FOs then there exists a stage $i >1$ such that  $\sum_{t \in T_i} |Q_i^t|\leq i /2$ (otherwise the overall number of FOs, $\sum_{i \in [n]}\sum_{t \in T_i} |Q_i^t|\geq n^2 /4$). Using this specific index $i\in [n]$ we construct the following (random) sequence of functions $f_1,\ldots,f_n$ where each $f_i:[-1,1]^2 \mapsto \mathbb{R}$: 

\[f_\ell(x)= \left\{
\begin{array}{ll}
      ([x]_1)^2 + ([x]_2)^2 & \ell \neq i,k\\
            ([x]_1)^2 + ([x]_2)^2 + ([x]_1 - [x]_2)^2 & \ell = k\\
            ([x]_1-1)^2 + ([x]_{1})^2 + ([x]_{2})^2 & \ell = i\\

\end{array} 
\right. \]
where $k \sim \mathrm{Unif}(1,\ldots,i-1)$.

Before proceeding we remark that each function $f_\ell: [-1,1]^2 \mapsto \mathbb{R}$ is $G$-Lipschitz, $L$-smooth and $\mu$-strongly convex with $G,L,\mu = \mathcal{O}(1)$.

\begin{restatable}{corollary}{constants}\label{c:constants}
Each $f_\ell:[-1,1]^2 \mapsto \mathbb{R}$ is $2$-strongly convex,  $6$-smooth and $6\sqrt{2}$-Lipschitz.
\end{restatable}

\underline{Let the initial point of $\mathcal{A}$ be $\hat{x}_0 = (0,0) \in [-1,1]^2$.} 

Notice that $\nabla f_\ell(0,0) = (0,0)$ for each $\ell \neq i$. Since $\mathcal{A}$ is a natural first-order method then Definition~\ref{d:nat_method}
implies that $\mathcal{A}$ always remains that $(0,0)$ at all stages $j \leq i-1$. The latter is formally stated and established in Lemma~\ref{c:11}.
\begin{restatable}{lemma}{celeven}\label{c:11}
For any stage $j \in [i-1]$,
\begin{itemize}
\item $[\hat{x}_j]_{1} = [\hat{x}_j]_{2} =0$
\item $[x_j^t]_{1} = [x^t_j]_{2} =0$ for all rounds $t \in [T_j]$.
\end{itemize}
\end{restatable}

The proof of Lemma~\ref{c:11} lies in Section~\ref{s:omit}. Its proof is based on the simple observation that if $\mathcal{A}$ has stayed at $(0,0)$ at all stages $\ell \leq i-2$. Then at stage $\ell+1 \leq i-1$, Item~$1$ of Definition~\ref{d:nat_method} implies that $\mathcal{A}$ always queries $\nabla f_{j}(0,0)$ for some $j \leq i-1$ meaning that $\nabla f_{j}(0,0) = (0,0)$. Then Item~$2$ and~$3$ of Definition~\ref{d:nat_method} directly imply that $x_\ell^t = (0,0)$ and $\hat{x}_\ell = (0,0)$. 
 
\begin{corollary}\label{c:stratis}
With probability greater than $1/2$, the function $f_k$ is never queried during state $i$. In other words, $q_{\mathrm{index}}\neq k$ for all $q \in \cup_{t \in T_i} Q_i^t$. 
\end{corollary}
\begin{proof}
By the selection of stage $i \in [n]$, we know that $\sum_{t \in T_i} |Q_i^t|\leq i /2$.  Since $k$ was sampled uniformly at random in $\{1,\ldots,i-1\}$, $\mathrm{Pr}\left[ q_{\mathrm{index}} \neq k~~\text{for all }q_\in \cup_{t \in T_i}Q_i^t \right] \geq 1/2$.
\end{proof}

Notice that for all $\ell \neq i$, $\nabla f_\ell(\cdot,0) = (\cdot , 0)$. The main idea of the proof is that in case $k$ is never queried during stage $i$, ($q_{\mathrm{index}}\neq k$ for all $q \in \cup_{t \in T_i} Q_i^t$) then all FOs during stage admit the form $\nabla f_{\mathrm{index}}(q_{\mathrm{value}}) = \nabla f_{\mathrm{index}}(\cdot,0) = (\cdot,0)$. The latter implies that if $q_{\mathrm{index}}\neq k$ for all $q \in \cup_{t \in T_i} Q_i^t$ then $[\hat{x}_i]_2 = 0$.

\begin{restatable}{lemma}{lthree}\label{l:29}
In case $q_{\mathrm{index}}\neq k$ for all $q \in \cup_{t \in T_i} Q_i^t$ then $[\hat{x}_{i}]_{2} =0$. 
\end{restatable}

The proof of Lemma~\ref{l:29} is based on the intuition that we presented above. Its proof is presented in Appendix~\ref{s:omit}.

To this end we are ready to complete the proof of Theorem~\ref{thm:lowerbound}. Combining Lemma~\ref{l:29} with Corollary~\ref{c:stratis} we get that
\begin{equation}\label{eq:prob}
\mathrm{Pr}\left[[\hat{x}_i]_{2} = 0\right] \geq 1/2.
\end{equation}

Up next we lower bound the error $g_i(\hat{x}_i) - g_i(x^\star_i)$ in case $[\hat{x}_i]_{2} =0$. 

By the definition of the sequence $f_1,\ldots,f_k,\ldots, f_i$, we get that
\begin{eqnarray*}
 g_i(x) &:=& \left([x]_1\right)^2 + \frac{1}{i}\left([x]_1 - 1\right)^2 + \frac{1}{i}\left([x]_1 - [x]_2\right)^2 + \left([x]_2\right)^2\\
\end{eqnarray*}

In case $[\hat{x}_i]_{2} =0$ we get that
\begin{eqnarray*}
g_i(\hat{x_i}) &=& \left(1+\frac{1}{i}\right)\left([\hat{x}_i]_1\right)^2+\frac{1}{i}\left([\hat{x_i}]_1-1\right)^2\\
 &\geq& \min_{w\in [-1,1]}\left[\left(1+\frac{1}{i}\right)w^2+\frac{1}{i}(w-1)^2 \right]~~~\text{notice that } w^\star = \frac{1}{i+2}\\
 &=& \frac{i+1}{i(i+2)^2} + \frac{(i+1)^2}{i(i+2)^2} = \frac{i+1}{i(i+2)}
\end{eqnarray*}
At the same time, the minimum $g_i(x_i^\star)$ equals,
\begin{eqnarray*}
  g_i(x^\star_i)&=& ([x_i^\star]_1)^2 + \frac{1}{i}\left([x_i^\star]_1 - 1\right)^2 + \frac{1}{i}\left([x_i^\star]_1 - [x_i^\star]_2\right)^2 + ([x_i^\star]_2)^2\
  \\
  &=& \min_{w,z\in [-1,1]}\left[w^2 + \frac{1}{i}\left(w - 1\right)^2 + \frac{1}{i}\left(w - z\right)^2 + z^2\right]~~~\text{notice that } (w^\star,z^\star) = \left(\frac{i+1}{i^2+3i+1}, \frac{1}{i^2+3i+1}\right)\\
  &=& \left(\frac{i+1}{i^2 + 3i +1}\right)^2 + \frac{1}{i}\left( \frac{i+1}{i^2 + 3i +1} -1\right)^2 + \frac{1}{i}\left(\frac{i+1}{i^2 + 3i +1} - \frac{1}{i^2 + 3i +1} \right)^2 + \left(\frac{1}{i^2 + 3i +1}\right)^2
\end{eqnarray*}
By taking the difference $g_i(\hat{x}_i) - g_i(x^\star_i)$ we get that
\begin{eqnarray*}
g_i(\hat{x_i}) - g_i(x^\star_i) &\geq& \min_{w\in [-1,1]}\left[w^2+\frac{1}{i}(w-1)^2 \right] - \min_{w,z\in [-1,1]^2}\left[\left(1-\frac{1}{i}\right)w^2 + \frac{1}{i}\left(w - 1\right)^2 + \frac{1}{i}\left(w - z\right)^2\right]\\
&=& \frac{i+1}{i(i+2)} - \left(\frac{i+1}{i^2 + 3i +1}\right)^2 - \frac{1}{i}\left( \frac{i+1}{i^2 + 3i +1} -1\right)^2 - \frac{1}{i}\left(\frac{i+1}{i^2 + 3i +1}- \frac{1}{i^2 + 3i +1} \right)^2\\
&-& \left(\frac{1}{i^2 + 3i +1}\right)^2\\
&=& \frac{1}{i^4 + 5i^3 + 7i^2 + 2i}
\end{eqnarray*}
We finally get that
\begin{eqnarray*}
\bE\left[ g_i(\hat{x_i}) - g_i(x^\star_i) \right] &=& \mathrm{Pr}\left[[\hat{x}_i]_{2}=0 \right] \cdot \bE\left[ g_i(\hat{x_i}) - g_i(x^\star_i)~~|~~ [\hat{x}_i]_{2}=0 \right]\\
&+& \mathrm{Pr}\left[[\hat{x}_i]_{2} \neq 0 \right] \cdot \bE\left[ g_i(\hat{x_i}) - g_i(x^\star_i)~~|~~ [\hat{x}_i]_{2} \neq 0 \right]\\
&\geq& \frac{1}{2} \cdot \frac{1}{i^4 + 5i^3 + 7i^2 + 2i}\\
&\geq& \Omega\left(1/n^4\right)
\end{eqnarray*}
\end{proof}
\subsection{Omitted Proofs}\label{s:omit}

\constants*
\begin{proof}
Notice that Hessians
\begin{itemize}
\item 
$\nabla^2 f_\ell(x) = \begin{pmatrix}
2 & 0 \\
0 & 2
\end{pmatrix}$ for $\ell \neq i,k$. Meaning that $f_\ell$ is $2$-strongly convex and $2$-smooth. 

\item 
$\nabla^2 f_k(x) = \begin{pmatrix}
4 & -2 \\
-2 & 4
\end{pmatrix}$ for $\ell \neq i,k$. Meaning that $f_\ell$ is $2$-strongly convex and $6$-smooth.

\item 
$\nabla^2 f_i(x) = \begin{pmatrix}
4 & 0 \\
0 & 2
\end{pmatrix}$ for $\ell \neq i,k$. Meaning that $f_\ell$ is $2$-strongly convex and $4$-smooth.
\end{itemize}

Respectively notice that the gradients 

\begin{itemize}
\item 
$\nabla f_\ell(x) = 2 ([x]_1,[x]_2)$ and thus $\max_{[x]_1,[x]_2 \in [-1,1]}\norm{\nabla f_\ell(x)}_2 \leq  2\sqrt{2}$. As a result, $f_\ell(x)$ is $2\sqrt{2}$-Lipschtiz in $[-1,1]^2$.

\item 
$\nabla f_k(x) = 2 (2 [x]_1 - [x]_2,2[x]_2 - [x]_1)$ and thus $\max_{[x]_1,[x]_2 \in [-1,1]}\norm{\nabla f_k(x)}_2 \leq  6\sqrt{2}$. As a result, $f_k(x)$ is $6\sqrt{2}$-Lipschtiz in $[-1,1]^2$.

\item 
$\nabla f_i(x) = 2 (2 [x]_1 - 1,2[x]_2)$ and thus $\max_{[x]_1,[x]_2 \in [-1,1]}\norm{\nabla f_i(x)}_2 \leq  4$. As a result, $f_k(x)$ is $4$-Lipschtiz in $[-1,1]^2$.
\end{itemize}
\end{proof}

\celeven*
\begin{proof}[Proof of Lemma~\ref{c:11}]
By the construction of the sequence $f_1,\ldots,f_n$, we get that for all $\ell \leq i-1$, 
\begin{itemize}
\item $\nabla f_\ell(x) =  2\left([x]_1,[x]_2\right)$ for $\ell \neq k,i$
\item $\nabla f_k(x) = 2 \left(2[x]_1 - [x]_2, 2[x]_2 - [x]_1\right)$
\end{itemize}

Thus in case $x= (0,0)$ then 
$\nabla f_\ell(x) = \nabla f_k(x) = (0,0)$.

Up next we inductively establish Lemma~\ref{c:11}. For the induction base we first establish that at stage $\ell:=1$,

\begin{corollary}\label{c:289}
For stage $\ell =1$,    
\begin{itemize}
\item $x^t_1 = (0,0)$ for all rounds $t \in [T_1]$.
\item $\hat{x}_1 =(0,0)$

\end{itemize}
\end{corollary}
\begin{proof}
At round $t:=1$ of stage $\ell:=1$, Item~1 of Definition~\ref{d:nat_method} ensures that for all FOs $q \in Q_1^1$ 
\[q_{\mathrm{index}} = 1 \text{ and } q_{\mathrm{value}} = (0,0)\]
Since $\nabla f_1(x) = 2x$, the latter implies that $\nabla f_{q_{\mathrm{index}}}(q_{\mathrm{value}}) = (0,0)$ for all $q \in Q_1^1$. Thus Item~$2$ of Definition~\ref{d:nat_method} ensures that $x_1^1 = (0,0)$.

We prove Corollary~\ref{c:289} through induction.

\begin{itemize}
\item \underline{Induction Base:} $x_1^1 = (0,0)$
\item \underline{Induction Hypothesis:} $x_1^{\tau} =(0,0)$ for all $\tau \in \{1, \ldots, t-1\}$.
\item \underline{Induction Step:} $x_1^t = (0,0)$
\end{itemize}

Item~$1$ of Definition~\ref{d:nat_method} ensures that for all $q \in Q_1^t$, $q_\mathrm{value} \in \left(\cup_{\tau \leq t-1} x_1^\tau\right) \cup \hat{x}_0$ and thus by the inductive hypothesis $q_\mathrm{value} = (0,0)$. Thus
\[\nabla f_{q_\mathrm{index}}(q_{\mathrm{value}}) =(0,0) ~~\text{for all }q =(q_\mathrm{index},q_{\mathrm{value}})  \in \cup_{\tau\leq t} Q_1^\tau \] 
Then Item~$2$ of Definition~\ref{d:nat_method} implies that $x_1^t=(0,0)$.

To this end we have established that $x_1^t = (0,0)$ for all $t \in [T_1]$. Then Item~$3$ of Definition~\ref{d:nat_method} implies that $\hat{x}_1 = (0,0)$.
\end{proof}

We complete the proof of Lemma~\ref{c:11} with a similar induction.

\begin{itemize}
\item \underline{Induction Base:} $x_1^{t} =(0,0)$ for all $t \in [T_1]$ and $\hat{x}_1 = (0,0)$.
\item \underline{Induction Hypothesis:} $x_\ell^{t} =(0,0)$ for all $t \in [T_\ell]$ and $\hat{x}_\ell = (0,0)$, for all stages $\ell \leq i-2$. 
\item \underline{Induction Step:} $x_{\ell+1}^{t} =(0,0)$ for all $t \in [T_{\ell+1}]$ and $\hat{x}_{\ell+1} = (0,0)$
\end{itemize}
Let us start with round $t :=1$ of stage $\ell+1$. Item~$1$ of Definition~\ref{d:nat_method} together with the inductive hypothesis ensure that for any FO $q \in Q_{\ell+1}^1$, $q_{\mathrm{value}} = (0,0)$. Since $q_{\mathrm{index}} \leq \ell + 1 \leq i-1$ we get that
\[\nabla f_{q_\mathrm{index}}(q_\mathrm{value}) = (0,0)~~~\text{for all queries }q \in Q_{\ell+1}^1.\]
The latter together with Item~$2$ of Definition~\ref{d:nat_method} imply 
\[x^1_{\ell+1} = (0,0).\]

Let us inductively assume that $x_{\ell+1}^t = (0,0)$. Then we the exact same argument as before we get that
\[\nabla f_{q_\mathrm{index}}(q_\mathrm{value}) = (0,0)~~~\text{for all queries }q \in Q_{\ell+1}^{t+1}.\]
Then again Item~$2$ of Definition~\ref{d:nat_method} implies that
\[x_{\ell + 1}^{t+1} = (0,0).\]
To this end we have established that $x_{\ell+1}^t = (0,0)$ for all $t \in [T_{\ell+1}]$. Thus Item~$3$ of Definition~\ref{d:nat_method} implies that $\hat{x}_{\ell+1} = (0,0)$. The latter completes the induction step and the proof of Lemma~\ref{c:11}.
\end{proof}

\lthree*

\begin{proof}[Proof of Lemma~\ref{l:29}]
First notice that 
\begin{itemize}
\item $\nabla f_\ell(x) =  2\left([x]_1,[x]_2\right)$ for $\ell \neq k,i$
\item $\nabla f_i(x) = 2(2[x]_1-1,[x]_2)$
\end{itemize}
Thus for any $x = \left([x]_1, 0\right)$, $\nabla f_{\ell}(x) = 2([x]_1,0)$ and $\nabla f_i(x) = 2(2[x]_1-1,0)$. Using the latter observation and the Lemma~\ref{c:11} we inductively establish that in case $q_{\mathrm{index}}\neq k$ for all $q \in \cup_{t \in T_i} Q_i^t$ then
\[x_i^t = \left( [x_i^t]_1, 0 \right).\]

We start by establishing the latter for $t = 1$. 
\begin{corollary}
In case $q_{\mathrm{index}}\neq k$ for all $q \in \cup_{t \in T_i} Q_i^t$ then
$x_i^t = \left( [x_i^t]_1, 0 \right)$.
\end{corollary}
\begin{proof}
Lemma~\ref{c:11} ensures that for all stages $\ell \leq i-1$
\begin{itemize}
\item $x^t_{\ell} =(0,0)$ for all $t \in [T_{\ell}]$.
\item $\hat{x}_{\ell} =(0,0)$
\end{itemize}
The latter together with Item~$1$ of Definition~\ref{d:nat_method} imply that $q_{\mathrm{value}} = (0,0)$ for all $q \in Q_i^1$. Since $q_{\mathrm{index}}\neq k$ for all $q \in Q_i^1$ we get that
\[\nabla f_{q_{\mathrm{index}}}(q_\mathrm{value}) = (\cdot , 0) \text{ for all }q \in Q_i^1\]
Then Item~$2$ of Definition~\ref{d:nat_method} implies that $x_i^1 = (\cdot , 0)$.
\end{proof}
We complete the proof of Lemma~\ref{l:29} through an induction. 

\begin{itemize}
\item \underline{Induction Base:} $x_i^1 = (\cdot , 0)$
\item \underline{Induction Hypothesis:} $x_i^\tau = (\cdot ,0)$ for all $\tau \leq t-1$
\item \underline{Induction Step:} $x_i^{t+1} = (\cdot ,0)$
\end{itemize}
The induction hypothesis together with Lemma~\ref{c:11} and Item~$1$ of Definition~\ref{d:nat_method} imply that for each FO $q \in Q_i^t$, $q_{\mathrm{value}} = (\cdot,0)$. Since $k \neq q_{\mathrm{value}}$ for all $q \in Q_i^t$, we get that
\[\nabla f_{q_{\mathrm{index}}}(q_\mathrm{value}) = (\cdot, 0) \text{ for all }q \in Q_i^t\] 
Then Item~2 of Definition~\ref{d:nat_method} implies that $x_i^t = (\cdot ,0)$.

To this end our induction is completed and we have established that $x_i^t = (\cdot ,0)$ for all $t \in [T_i]$. Then the fact that $\hat{x}_i = (\cdot ,0)$ is directly implied by Item~$3$ of Definition~\ref{d:nat_method}.

\end{proof}
\newpage

\section{Sparse Stochastic Gradient Descent}
In this section we present the sparse version Stochastic Gradient Descent, called $\mathrm{SGD}$-$\mathrm{Sparse}$ (Algorithm~\ref{alg:sparseSGD}) that is able to ensure accuracy $\epsilon >0$ at each stage $i \in [n]$ with $\mathcal{O}\left(\frac{\mathcal{D}G^3 \log n}{\epsilon^2}\right)$ FO complexity. 

In Algorithm~\ref{alg:SGD} we first present SGD for specific stage $i \in [n]$.

\begin{algorithm}[ht]
    \caption{$\mathrm{StochasticGradientDescent}(i,x,\gamma,T_i)$}
    \label{alg:SGD}
\begin{algorithmic}[1]
    \State $x_0^i \leftarrow x \in \mathcal{D}$
    \For{each round $t:= 1,\dots,T_i$}
    \smallskip
            \State Sample $j_t \sim \mathrm{Unif}(\{1,\ldots,i-1\})$
            \smallskip
            \State $x_i^t \leftarrow \Pi_{\mathcal{D}}\left( x_i^{t-1} - \gamma \nabla f_{j_t}(x_i^{t-1})/t\right)$
            \smallskip
    \EndFor
    \smallskip
    \State \textbf{return} $\hat{x}_i \leftarrow \sum_{t=1}^{T_i} x_i^t / T_i$
    \end{algorithmic}
\end{algorithm}

$\mathrm{SGD}$-$\mathrm{sparse}$ is presented in Algorithm~\ref{alg:sparseSGD}. Algorithm~\ref{alg:sparseSGD} tracks a sparse sub-sequence of stages $i \in [n]$ ($\prev \cdot (1 + \alpha) < i$) at which the returned point $\hat{x}_i$ is produced by running $\mathrm{SGD}$ (Algorithm~\ref{alg:SGD}) for $T_i$ iterations. In the remaining epochs Algorithm~\ref{alg:sparseSGD} just returns the solution of the last stage at which $\mathrm{SGD}$ was used ( $\hat{x}_i \leftarrow \hat{x}_{\prev}$, Step~$8$ of Algorithm~\ref{alg:sparseSGD}). Notice that the sparsity of the SGD-call (Step~$4$) is controlled by the selection of the parameter $\alpha \geq 0$. For example for $\alpha = 0$, Step~$4$ is reached at each stage $i \in [n]$ while for large values of $\alpha \geq 0$, Step~$4$ might never be reached.   

\begin{algorithm}[ht]
    \caption{$\mathrm{StochasticGradientDescent}$-$\mathrm{sparse}$}
    \label{alg:sparseSGD}
\begin{algorithmic}[1]
 \State $\hat{x}_0 \in \mathcal{D}$ and $\prev\leftarrow 0$
    \smallskip
    \For{each stage $i:= 1,\dots,n$}
        \smallskip
        \If{$\prev\cdot(1+\alpha)<i$}
            \smallskip
            \State $\hat{x}_i \leftarrow \mathrm{StochasticGradientDescent}(i,\hat{x}_{i-1}, \gamma, T_i)$~~~\textcolor{blue}{$\#$\textit{ output $\hat{x}_i \in \mathcal{D}$ for stage $i$}}
            \smallskip
            \State $\hat{x}_{\prev} \leftarrow \hat{x}_i$
        \smallskip
        \State $\prev \leftarrow i $
        \Else
            \smallskip
            \State $\hat{x}_i \leftarrow \hat{x}_{\prev}$~~~\textcolor{blue}{$\#$\textit{ output $\hat{x}_i \in \mathcal{D}$ for stage $i$}}
            \smallskip
        \EndIf
    \EndFor
    \end{algorithmic}
\end{algorithm}
In the rest of the section we establish the formal guarantees of $\mathrm{SGD}$ that are formally stated in Algorithm~\ref{alg:sparseSGD}. We start by presented the well-known guarantees of $\mathrm{SGD}$. 
\begin{theorem}[\cite{H16}]\label{t:SGD}
Let $\hat{x}_i := \mathrm{StochasticGradientDescent}(i,\hat{x}_{i-1},\gamma,T_i)$ where $\gamma = 1/\mu$, $T_i = \mathcal{O}\left(\frac{G^2}{\epsilon \mu}\right)$ and $G = \max_{x \in \mathcal{D}}\norm{\nabla f(x)}$. Then, $\mathrm{E}\left[ g_i(\hat{x}_i) - g_i(x_i^\star) \right] \leq \epsilon$.
\end{theorem}

\begin{theorem}
Let a convex and compact set $\mathcal{D}$ and a sequence of $\mu$-strongly convex functions $f_1,\ldots,f_n$ with $f_i:\mathcal{D} \mapsto \mathbb{R}$. Then Algorithm~\ref{alg:sparseSGD}, with $T_i = \mathcal{O}\left(\frac{G^2}{\epsilon \mu}\right)$, $\alpha = \frac{1}{2|\mathcal{D}|G}$ and $\gamma = 1/\mu$, guarantees
\[\mathrm{E}\left[ g_i(\hat{x}_i) - g_i(x_i^\star) \right] \leq \epsilon \text{ for each stage } i \in [n] \]
The FO complexity of Algorithm~\ref{alg:sparseSGD} is $\mathcal{O}\left(\frac{|\mathcal{D}|G^3 \log n}{\epsilon^2 \mu}\right)$.
\end{theorem}

\begin{proof}
In case $i = \prev$ then by selecting $T = \mathcal{O}\left(\frac{G^2}{\mu \epsilon}\right)$ and $\gamma = 1\ / \mu$, Theorem~\ref{t:SGD} implies that $\bE\left[g_i(\hat{x}_{i})\right]-g_i(x^\star_i) \leq \epsilon / 2$. Up next we establish the claim for $\prev < i \leq (1+\alpha ) \cdot \prev$. 
\begin{eqnarray*}
\bE\left[g_i(\hat{x}_{i})\right]-g_i(x^\star_i) &=& \bE\left[g_i(\hat{x}_{\prev})\right]-g_i(x^\star_i) ~~~(\hat{x}_i = \hat{x}_{\prev},~\text{Step }8 \text{ of Algorithm~\ref{alg:sparseSGD}})\\
&=& \bE\left[\frac{1}{i}\sum_{j=1}^i f_j(\hat{x}_{\prev})\right]-\frac{1}{i}\sum_{j=1}^if_j(x^\star_i)\\
&=& \frac{\prev}{i} \cdot \bE\left[g_{\prev}(\hat{x}_{\prev})\right] + \frac{1}{i}\sum_{j=\prev+1}^i\bE\left[f_j(\hat{x}_{\prev})\right] - \frac{\prev}{i} \cdot g_{\prev}(x^\star_i) - \frac{1}{i}\sum_{j=\prev+1}^i f_j(x^\star_i)\\
&=& \frac{\prev}{i} \cdot \left(\bE\left[g_{\prev}(\hat{x}_{\prev})\right]-g_{\prev}(x^\star_i)\right) + \frac{1}{i}\sum_{j=\prev+1}^j\left(\bE\left[f_j(\hat{x}_{\prev})\right]-f_j(x^\star_i)\right)\\
&\leq& \frac{\prev}{i} \cdot \left(\bE\left[g_{\prev}(\hat{x}_{\prev})\right]-g_{\prev}(x^\star_i)\right) + \frac{i-\prev}{i} \cdot |\mathcal{D}|G\\
&=& \frac{\prev}{i}\left(\bE\left[g_{\prev}(\hat{x}_{\prev})\right]-g_{\prev}(x^\star_{\prev})+\underbrace{g_{\prev}(x^\star_{\prev})-g_{\prev}(x^\star_i)}_{\leq 0}\right) + |\mathcal{D}|G \cdot \frac{i-\prev}{\prev}\\
        &\leq& \frac{\prev}{i}\left(\bE\left[g_{\prev}(\hat{x}_{\prev})\right]-g_{\prev}(x^\star_{\prev})\right) + |\mathcal{D}|G\cdot \frac{i-\prev}{\prev}\\
        &\leq& \frac{\epsilon}{2} + |\mathcal{D}|G\cdot \frac{i-\prev}{\prev}\\
        &\leq& \frac{\epsilon}{2} + |\mathcal{D}|G\frac{(1+\alpha)\prev-\prev}{\prev}\\
        &=& \frac{\epsilon}{2} + |\mathcal{D}|G \alpha = \epsilon
\end{eqnarray*}
To this end we have established that $\bE\left[g_{i}(\hat{x}_{i})\right]-g_{i}(x^\star_i) \leq \epsilon$ for each stage $i \in [n]$. We conclude the proof by upper bounding the overall FO complexity of Algorithm~\ref{alg:sparseSGD}.

An execution of Stochastic Gradient Decent at Step~4 of Algorithm~\ref{alg:sparseSGD} that calculates an $\epsilon/2$ optimal solution requires $\cO\left(G^2/(\mu\epsilon)\right)$ FOs. Let us now upper bound the required executions of SGD. Let $k$ denotes number of times Algorithm~\ref{alg:sparseSGD} reached Step~$4$ over the course of $n$ stages. Due Step~5 of Algorithm~\ref{alg:sparseSGD},
\[ \left(1+\frac{\epsilon}{2|\mathcal{D}|G}\right)^k\leq n \Rightarrow k \leq    \frac{\ln n}{\ln (1+\frac{\epsilon}{2|\mathcal{D}|G})} \leq \frac{\ln n}{1-\frac{1}{1+\frac{\epsilon}{2|\mathcal{D}|G}}} \leq 4|\mathcal{D}|G\frac{\log n}{\epsilon}\]
Thus the overall FO complexity of Algorithm~\ref{alg:sparseSGD} is $\mathcal{O}\left( \frac{G^3 |\mathcal{D}| \log n}{\mu \epsilon^2} \right)$
\end{proof}
\clearpage

\label{app:lowaccuracy}

\section{Experimental Details}
In this section we present additional experimental evaluations in LIBSVM datasets. 

\begin{table}[h!]
\centering
\begin{tabular}{cc}
\includegraphics[width=0.3\textwidth]{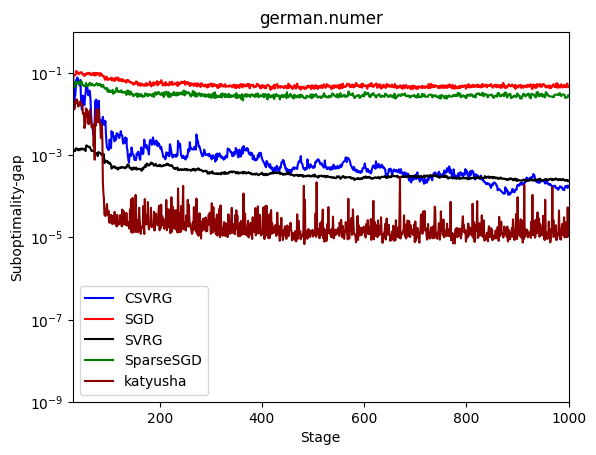} & \includegraphics[width=0.3\textwidth]{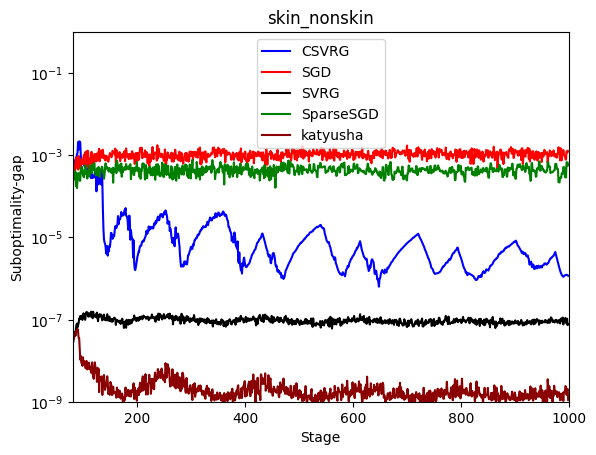}\\ \includegraphics[width=0.3\textwidth]{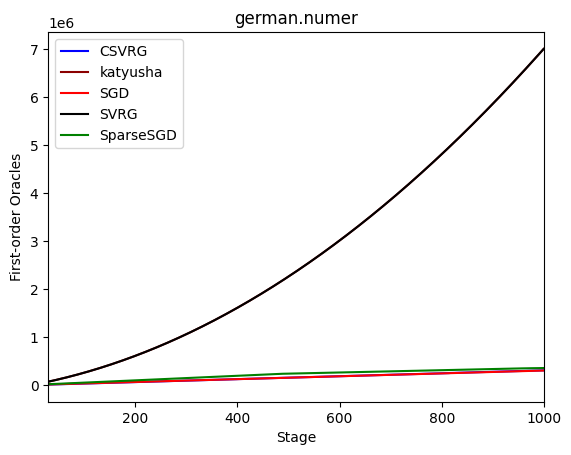} & \includegraphics[width=0.3\textwidth]{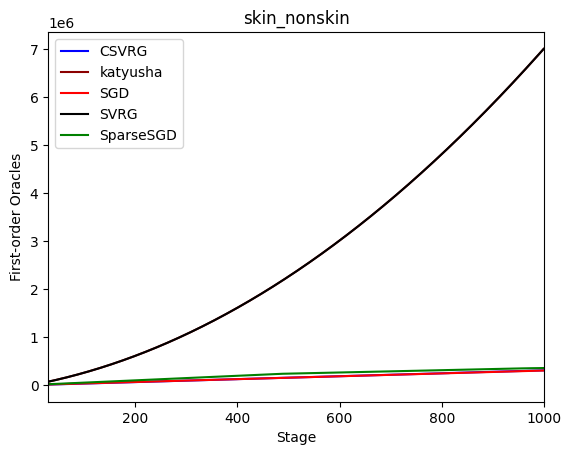}
\end{tabular}
\caption{Optimality gap as the stages progress on a ridge regression problem (averaged over $10$ independent runs). $\texttt{CSVRG}$ performs the exact same number of FOs with $\texttt{SGD}$ and slightly less than $\texttt{SGD}$-$\texttt{sparse}$. $\texttt{Katyusha}$ and $\texttt{SVRG}$ perform the exact same number of FOs. $\texttt{CSVRG}/\texttt{SGD}/\texttt{SGD}$-$\texttt{sparse}$ perform roughly $4\%$ of the FOs of $\texttt{Katyusha}/\texttt{SVRG}$.}
\label{tab:error}
\end{table}
In Table~4 and~5 we present the parameters of the various method used in our experimental evaluations. We remind that $\lambda = 10^{-3}$.
\begin{table}[!ht]
    \begin{center}
    \begin{tabular}{|c||p{35mm}|p{35mm}|p{35mm}|}
        \hline
        Method & \textbf{breast cancer} & \textbf{cod-rna} & \textbf{diabetes}\\
        \hline
        \texttt{SGD} & step size: $(t\lambda)^{-1}$ \newline  $T$: 300  &  step size: $(t\lambda)^{-1}$ \newline $T$: 300&step size: $(t\lambda)^{-1}$ \newline $T$: 300 \\
        \hline
         \texttt{SGD}-\texttt{sparse} & step size: $(t\lambda)^{-1}$\newline $T$: 414 \newline $\alpha:0.002$
          & step size: $(t\lambda)^{-1}$\newline $T$: 480 \newline $\alpha:0.002$ & step size: $(t\lambda)^{-1}$\newline $T$: 480 \newline $\alpha:0.002$\\
        \hline
        \texttt{Katyusha} & step size: $1/(3L)$ \newpage Outer Iterations: 10 \newpage Inner Iterations: 100 \newpage $L:~0.0522 $  & step size: $1/(3L)$ \newpage Outer Iterations: 10 \newpage Inner Iterations: 100 \newpage $L:~0.015$ & step size: $1/(3L)$ \newpage Outer Iterations: 10 \newpage Inner Iterations: 100 \newpage $L:~0.01262$\\
        \hline
                \texttt{SVRG} & step size: $1/(3L)$ \newpage Outer Iterations: 10 \newpage Inner Iterations: 100 \newpage $L:~0.0522$  & step size: $1/(3L)$ \newpage Outer Iterations: 10 \newpage Inner Iterations: 100 \newpage $L:~0.015$ & step size: $1/(3L)$ \newpage Outer Iterations: 10 \newpage Inner Iterations: 100 \newpage $L:~0.01262$\\
        \hline
        \texttt{CSVRG} & step size at $i$: $(it\lambda)^{-1}$ \newpage $T_i = 100$ \newpage$\alpha = 0.3$  & step size at $i$: $(it\lambda)^{-1}$ \newpage $T_i = 100$ \newpage$\alpha = 0.3$ & step size at $i$: $(it\lambda)^{-1}$ \newpage $T_i = 100$ \newpage$\alpha = 0.3$ \\
        \hline
    \end{tabular}
        \caption{Parameters used in our experiments}
    \end{center}
    \label{tab:params1}
\end{table}

\begin{table}[t]
    \begin{center}
    \begin{tabular}{|c||p{35mm}|p{35mm}|p{35mm}|}
        \hline
        Method & \textbf{german.numer} & \textbf{skin-nonskin}\\
        \hline
        \texttt{SGD} & step size: $(t\lambda)^{-1}$ \newline  $T$: 300  &  step size: $(t\lambda)^{-1}$ \newline $T$: 300\\
        \hline
         \texttt{SGD}-\texttt{sparse} & step size: $(t\lambda)^{-1}$\newline $T$: 414 \newline $\alpha:0.002$
          & step size: $(t\lambda)^{-1}$\newline $T$: 480 \newline $\alpha:0.002$\\
        \hline
        \texttt{Katyusha} & step size: $1/(3L)$ \newpage Outer Iterations: 10 \newpage Inner Iterations: 100 \newpage $L:~0.0317$  & step size: $1/(3L)$ \newpage Outer Iterations: 10 \newpage Inner Iterations: 100 \newpage $L:~0.068$\\
        \hline
                \texttt{SVRG} & step size: $1/(3L)$ \newpage Outer Iterations: 10 \newpage Inner Iterations: 100 \newpage $L:~0.0317$  & step size: $1/(3L)$ \newpage Outer Iterations: 10 \newpage Inner Iterations: 100 \newpage $L:~0.068$\\
        \hline
        \texttt{CSVRG} & step size at $i$: $(it\lambda)^{-1}$ \newpage $T_i = 100$ \newpage$\alpha = 0.3$  & step size at $i$: $(it\lambda)^{-1}$ \newpage $T_i = 100$ \newpage$\alpha = 0.3$\\
        \hline
    \end{tabular}
    \caption{Parameters used in our experiments}
    \end{center}
    \label{tab:params2}
\end{table}

\clearpage

\section{Experimental Evaluations with Neural Networks}
In this section we include additional experimental evaluations \texttt{CSVRG},\texttt{SGD} and \texttt{SVRG} for the continual finite-sum setting in the context of $2$-Layer Neural Networks and the MNIST dataset. More precisely, we experiment with Neural Networks with an input layer of size $784$ (corresponds to the size of the inputs from the MNIST dataset), a hidden layer with $1000$ nodes and an ReLU activation function and an output layer with 10 nodes corresponding to the $10$ MNIST classes. As a training loss we use the cross-entropy between the model's output and the one-hot encoding of the classes. We consider two different constructions of the data streams that are respectively presented in Section~\ref{s:nn1} and Section~\ref{s:nn2}. For both settings the parameters used for the various methods are presented in Table~\ref{tab:mnist_experiment}. We also remark that the parameters of \texttt{CSVRG}, \texttt{SGD} were appropriately selected so as to perform roughly the same number of FOs.
\begin{table}[h!]
    \begin{center}
    \begin{tabular}{|c||p{35mm}|}
        \hline
        Method & \textbf{MNIST}\\
        \hline
        \texttt{SGD} & step size: $0.001$ \newline  $T$: 1385 \\
        \hline
                \texttt{SVRG} & step size: $0.05$ \newpage Outer Iterations: 7 \newpage Inner Iterations: 40 \\
        \hline
        \texttt{CSVRG} & step size at $i$: $0.001$ \newpage $T_i = 600$ \newpage$\alpha = 0.01$\\
        \hline
    \end{tabular}
    \caption{Parameters used for training of the neural network.}
    \end{center}
    \label{tab:mnist_experiment}
\end{table}

\subsection{Data Streams produced with the Stationary Distribution}\label{s:nn1}

In this experiment we create a stream of $3000$ data points where the data-point at iteration $i$ were sampled uniformly at random (without replacement) from the MNIST dataset. Table~\ref{tab:MNIST} illustrates the training loss of each different method across the different stages. As Table~\ref{tab:MNIST} reveals that in the early stages all methods achieve training loss close to $0$ by overfitting the model while in the latter stages \texttt{CSVRG} and \texttt{SVRG} are able to achieve much smaller loss in comparison to \texttt{SGD}. We remark that \texttt{CSVRG} is able to meet the latter goal with way fewer FO calls than \texttt{SVRG} as demonstrated in Table~\ref{tab:MNIST}.

\begin{table}[!h]
    \centering
    \begin{tabular}{cc}
       \includegraphics[width=0.4\textwidth]{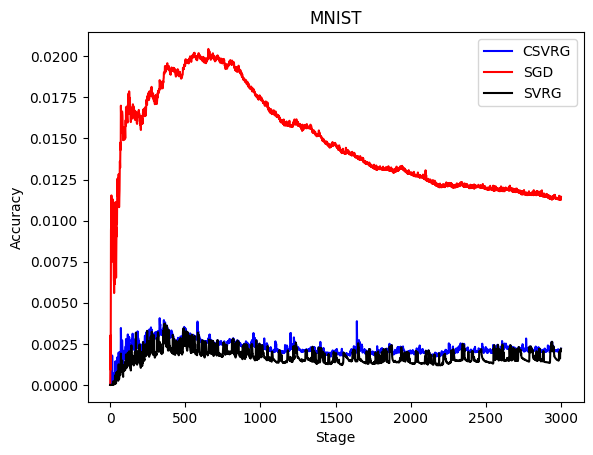}  &
       \includegraphics[width = 0.4\textwidth]{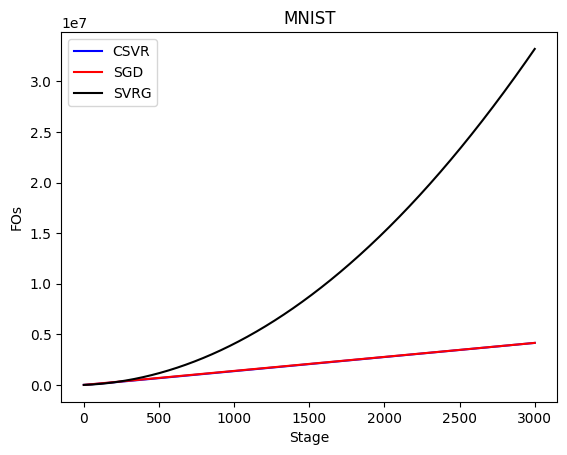}
    \end{tabular}
    \caption{Training loss as the stages progress on a neural network training for the MNIST dataset for the setting of sampling without replacement.}
    \label{tab:MNIST}
\end{table}

\subsection{Data Streams with Ascending Labels}\label{s:nn2}
Motivated by continual learning at which new tasks appear in a streaming fashion, we experiment with a sequence of MNIST data points at which \textit{new digits} gradually appear. More precisely, we construct the stream of data points as follows:
\begin{enumerate}
    \item For each of the classes $i \in \{0,9\}$, we randomly sample $300$ data points.
    \item For the first $600$ stages we randomly mix $0/1$ data points.
    \item The stages $\{601 + (i-2)*300, 600 + (i-1)*300\}$ for $i \in \{2,9\}$ contain the data points of category $i$ (e.g. category $2$ appears in stages $\{601,900\}$).
\end{enumerate}

Table~\ref{tab:error1} illustrates the training loss of \texttt{CSVRG} and \texttt{SGD} for the various stages of the above data-stream. As Table~\ref{tab:error1} reveals, \texttt{CSVRG} is able to achieve significantly smaller loss than \texttt{SGD} with the same number of FOs. We also remark that that the bumps appearing in bot curves corresponds to the stages at which a \textit{new digit} is introduced.

\begin{table}[h!]
\centering
\begin{tabular}{cc}
\includegraphics[width=0.4\textwidth]{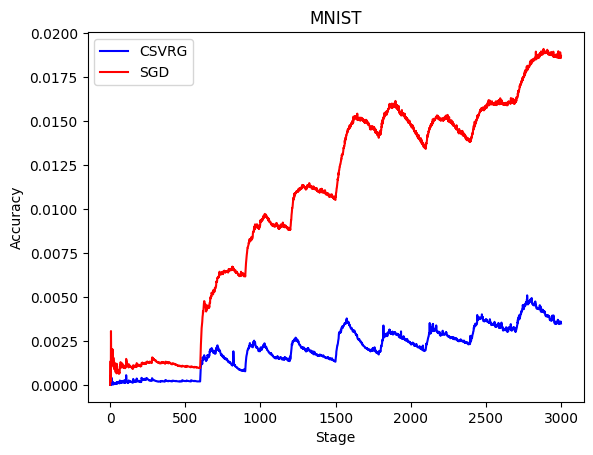} & 
\end{tabular}
\caption{The training loss of \texttt{CSVRG} and \texttt{SGD} for the various stages.}
\label{tab:error1}
\end{table}

In Table~\ref{tab:error2} we present the classification accuracy of the models respectively produced by \texttt{CSVRG} and \texttt{SGD} for the various stages. At each stage $i$ we present the accuracy of the produced model according to the categories revealed so far. The sudden drops in the accuracy occur again at the stages where new digits are introduced. For example in iteration $601$, the accuracy drops from roughly $1$ to roughly $0.667$ due to the fact in stage $600$ the accuracy is measured with respect to the $0/1$ classification task while in stage $601$ the accuracy is measured with respect to the $0/1/2$ classification task. At stage $601$ both models misclassifies all the $2$ examples and thus the accuracy drops to $2/3$. In the left plot, we plot the classification accuracy with respect to the data used in the training set while in the right one we plot the classification accuracy on unseen test data. Both plots reveal that \texttt{CSVRG} admits a noticable advantage of \texttt{SGD}.

\begin{table}[h!]
\centering
\begin{tabular}{cc}
\includegraphics[width=0.4\textwidth]{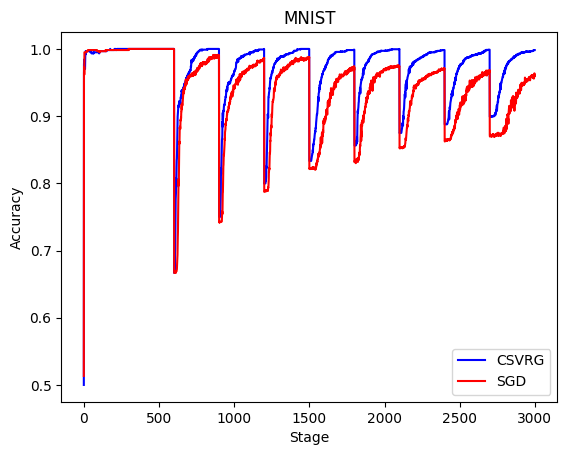}
\includegraphics[width=0.4\textwidth]{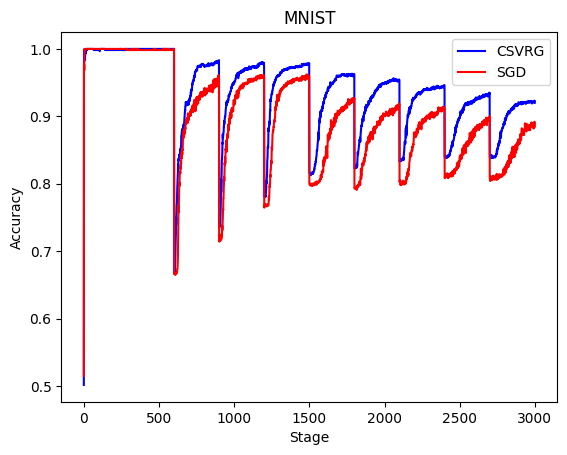} 
\end{tabular}
\caption{Classification accuracy. For the train set on the left and the test set on the right. As a comparison between the two algorithms.}
\label{tab:error2}
\end{table}

\label{app:expdetails}

\end{document}